\documentclass[preprint,12pt]{elsarticle}

\usepackage{amssymb}

\usepackage[mathscr]{eucal}
\usepackage[all]{xy} 
\usepackage{graphicx}
\usepackage{caption}
\usepackage{subcaption}
\usepackage{multicol}
\usepackage{amssymb,amsfonts,amsthm}    
\usepackage{color}                              
\usepackage[centertags]{amsmath}
\usepackage{amsfonts,amssymb,amsthm,newlfont}
\usepackage{geometry}
\usepackage{algorithmic,algorithm}
\usepackage{hyperref}
\usepackage{mathrsfs}
\usepackage{enumerate}

\usepackage{nomencl}

\theoremstyle{plain}
\newtheorem{thm}{Theorem}[section]

\newtheorem{lemma}[thm]{Lemma}

\newtheorem{definition}[thm]{Definition}
\newtheorem{remark}[thm]{Remark}

\numberwithin{equation}{section}
\numberwithin{algorithm}{section}

\def\te{\theta_{\epsilon}}
\def\tz{z_{\epsilon}}
\def\tr{\rho_{\epsilon}}

\begin{document}

\begin{frontmatter}



\title{An Approximating Control Design for Optimal  Mixing by Stokes  Flows}
\author{Weiwei Hu}


 \ead{weiwei.hu@okstate.edu}

\address{Department of Mathematics,
416 Math Science Building,
Oklahoma State University,
Stillwater, OK 74078}

\begin{abstract}
We consider   an approximating  control design  for  optimal mixing of
a non-dissipative scalar field $\theta$ in an unsteady Stokes flow.
The objective of our approach  is to achieve optimal mixing  at a given final time  $T>0$, via  the active control of the  flow velocity $v$ through boundary inputs.   Due to  zero diffusivity of the scalar field $\theta$, establishing the well-posedness of its G\^{a}teaux derivative requires 
$\sup_{t\in[0,T]}\|\nabla \theta\|_{L^2}<\infty$, which in turn demands the flow velocity field to satisfy the condition 
$ \int^{T}_{0}\|\nabla v\|_{L^{\infty}(\Omega)}\, dt<\infty$. This condition results in the need to penalize the time derivative of the boundary control in the cost functional.   Consequently,   the optimality system becomes  difficult to solve \cite{hu2017boundary}. 
Our current approximating  approach provides  a more transparent  optimality system, with  the set of admissible controls square integrable in space-time.  This is achieved  by first introducing a small diffusivity  to the scalar equation and then establishing  a rigorous analysis of  convergence of the approximating control  problem to the original one as the diffusivity approaches  to zero. Uniqueness of the optimal solution  is obtained for the two dimensional case.

\end{abstract}






\end{frontmatter}





\section{Introduction}
Consider a passive scalar field  advected by an  unsteady  Stokes  flow  on an open bounded and connected  domain $\Omega\subset \mathbb{R}^{d}$, where $d=2,3,$ with a sufficiently smooth boundary $\Gamma$. The scalar field is governed by the transport equation,  where the molecular diffusion is  assumed to be negligible and mixing is purely  driven by  advection. This naturally leads to  the study of optimal mixing  via an active control of the flow velocity. As discussed in   our previous work \cite{hu2017boundary},  we consider the flow velocity  induced  by   control inputs  acting tangentially  on the boundary of the domain through the  Navier slip boundary conditions. This is motivated by  the observation that moving walls accelerate mixing compared to fixed walls (cf.~\cite{GDDRT, GKDDRT,GTD, OG, TGD}).  We aim at  designing a Navier slip  boundary  control  that optimizes mixing at a given final time. The governing system of equations is 
  \begin{align}
      & \frac{\partial \theta}{\partial t}+ v\cdot \nabla \theta=0,  \label{EQ01}\\
&\frac{\partial v}{\partial t} - \Delta v+ \nabla p=0,   \label{Stokes1}\\
&\nabla \cdot v = 0, \quad x\in\Omega,  \label{Stokes2}
  \end{align}
with the Navier slip boundary conditions  (cf.~\cite{ hu2018boundary, kelliher2006navier,navier}),
      \begin{align}
   v\cdot n|_{\Gamma}=0 \quad   \text{and}   \quad kv+(\mathbb{T}(v)\cdot n)_{\tau} |_{\Gamma}=g, \label{Stokes3}
    \end{align}
    and the initial condition  is given by 
\begin{align}
(\theta(0), v(0))=(\theta_{0}, v_{0}),\label{ini}
\end{align} 
    where $\theta$ is the density, $v$ is the velocity,  $p$ is the pressure, and  $g$ is the boundary control input, which is employed  to  generate the velocity field for mixing.   Navier slip boundary conditions admit the fluid to slip with resistance on the boundary.  Here  $n$ and $\tau$ denote the outward unit normal and tangentially vectors with respect to the domain $\Omega$, $\mathbb{T}(v)=2 \mathbb{D}(v)$ with $\mathbb{D}(v)=(1/2)(\nabla v+(\nabla v)^{T})$, $(\mathbb{T}(v)\cdot n)_{\tau}$ denotes the tangential component of $(\mathbb{T}(v)\cdot n)$, and  $g\cdot n|_{\Gamma}=0$. The   friction between the fluid and the wall is proportional to $-v$ with the
positive coefficient of proportionality  $k$.  

Due to the divergence-free and no-penetration boundary conditions imposed on the  velocity field, it can be shown that any $L^p$-norm of $\theta$ is conserved (cf.~\cite{HKZ, hu2017boundary}), i.e.,
\begin{align}
\Vert\theta(t)\Vert_{L^{p}}= \Vert\theta_{0}\Vert_{L^{p}} , \quad t\geq 0, \quad p\in [0,\infty].  \label{EST_theta_Lp}
\end{align}
To qualify mixing,  the mix-norm and negative Sobolev norms $H^{-s}$, for any $s>0$, are usually adopted, especially for  the scalar field with  no molecular diffusion, based on ergodic theory (cf.~\cite{LTD, LLNMD, MMGVP, MMP, Th}).  The bridge that connects mixing with  negative Sobolev norms  is the  property of weak convergence. As discussed  in our previous work, we consider a general bounded domain for mixing and  replace the negative Sobolev norm  by the norm for the dual space $(H^{s}(\Omega))'$ of $H^{s}(\Omega)$ with $s>0$. Also, we identify the  space $(H^{s}(\Omega))'$, where $ s>0$,  as the domain of operator $\Lambda^{-s}$ equipped with the norm $\|\cdot\|_{(H^{s}(\Omega))'}$,  where $\Lambda $ is self-adjoint, positive and unbounded in $L^{2}(\Omega)$ (cf.~\cite[p.~9]{lions1971}). Thus,
 $\Lambda^{2s}\in \mathcal{L}(H^{s}(\Omega), (H^{s}(\Omega))')$.    In our current work, we continue to adopt $\|\cdot\|_{(H^1(\Omega))'}$
for  qualifying mixing as in \cite{hu2017boundary}.     
In particular, we choose $\Lambda=\mathcal{A}^{-1/2}$, where $\mathcal{A}$ is given by
\begin{align*}
  \mathcal{A}\phi=(-\Delta+I)\phi, \quad \phi  \in D(\mathcal{A})=\{\phi\in H^2(\Omega)\colon\frac{\partial \phi}{\partial n}|_{\Gamma}=0 \}.
    \end{align*} 
%
Then  $D(\Lambda)=D(\mathcal{A}^{1/2})=H^1(\Omega)$  and $D(\Lambda^{-1})=D(\mathcal{A}^{-1/2})=(H^1(\Omega))'$.

Throughout  this paper, we use $(\cdot, \cdot)$  and $\langle \cdot, \cdot \rangle$ for the $L^{2}$-inner products in 
the interior of the domain $\Omega$ and on the boundary $\Gamma$, respectively. 
To set up the abstract formulation for the velocity field, we
define 
\begin{align*}
V^{s}_{n}(\Omega)&=\{ v \in H^{s}(\Omega): \text{div}\ v=0,\  v\cdot n|_{\Gamma}=0\}, \quad  s\geq 0, \nonumber\\
V^{s}_{n}(\Gamma)&=\{ g \in H^{s}(\Gamma): g\cdot n|_{\Gamma}=0\}, \quad  s\geq 0.
\end{align*}

 The optimal control problem is  formulated as follows: For a given $T>0$, find  a control $g$ minimizing the cost functional 
\begin{align*}
J(g)=\frac{1}{2} \|\theta(T)\|^{2}_{(H^{1}(\Omega))'}
+\frac{\gamma}{2}  \| g\|^{2}_{U_{\text{ad}}}, \qquad (\text{P})
\end{align*}
subject to \eqref{EQ01}--\eqref{ini}, where $\|\theta(T)\|_{(H^{1}(\Omega))'}=\|\Lambda^{-1}\theta(T)\|_{L^{2}(\Omega)}$,  
 $\gamma>0$ is the control weight parameter, and $U_{\text{ad}}$ is   the set of admissible controls, which is often determined based on the physical properties as well as the need to establish the well-posedness of the problem, i.e., the existence of an optimal solution.     In fact, the existence of an optimal solution to the problem~$(P)$ can be proven for $U_{ad}=L^{2}(0, T; V^0_{n}(\Gamma))$.  The challenge arises in deriving the first-order necessary conditions of optimality. To establish the well-posedness of the G\^ateaux derivative of $\theta$, one needs $\sup_{t\in[0,T]}\|\nabla \theta\|_{L^2}<\infty,$ which requires 
$\theta_0\in H^{1}(\Omega)$ and  the flow velocity to satisfy   
\[ \int^{T}_{0}\|\nabla v\|_{L^{\infty}(\Omega)}\, dt<\infty.\]
 Therefore, the initial condition $v_0$ and $U_{\text{ad}}$ were chosen in a way such  that this estimate holds  \cite{hu2017boundary}. 
As a result, the time regularity of $g$ was needed. For computational convenience,  the first derivative $\partial{g}/\partial{t}$ was adopted rather than the  lower order  fractional time derivative  in the cost  functional.   Consequently,  the optimality condition involved the  time derivative of $g$, and  thus the optimality system became difficult to further analyze the  uniqueness of the solution.

\subsection{An approximating control approach }
In this work, we start with  investigating   the  approximating control problem by adding a small diffusion term $\epsilon \Delta \theta$, for $\epsilon>0$,  to the transport equation. The problem is now  formulated as follows:  For  a given $T>0$, find  a control $g_{\epsilon}\in U_{\epsilon_{\text{ad}}}=L^{2}(0, T; V^0_n(\Gamma))$ minimizing the cost functional 
\begin{align*}
J_{\epsilon}(g_{\epsilon})=\frac{1}{2} \|\te(T)\|^{2}_{(H^{1}(\Omega))'}
+\frac{\gamma}{2}  \| g_{\epsilon}\|^{2}_{U_{\epsilon_\text{ad}}}, \quad \epsilon>0,\qquad (\text{P}_{\epsilon})
\end{align*}
subject to  an approximating system governed by
  \begin{align}
   & \frac{\partial \te}{\partial t}-\epsilon \Delta \te+ v_{\epsilon}\cdot \nabla \te=0,  \label{App_EQ01}\\
&\frac{\partial v_{\epsilon}}{\partial t} - \Delta v_{\epsilon}+ \nabla p_{\epsilon}=0,   \label{App_Stokes1}\\
&\nabla \cdot v_{\epsilon} = 0, \quad x\in\Omega,  \label{App_Stokes2}
  \end{align}
 with the Neumann boundary condition  for the scalar 
        \begin{align}
      \epsilon \frac{\partial \te}{\partial n}\Bigm|_{\Gamma}= 0\label{App_BC_EQ01}
      \end{align}
and the nonhomogenous Navier slip boundary conditions for the velocity 
  \begin{align}
   v_{\epsilon}\cdot n|_{\Gamma}=0 \quad   \text{and}   \quad( kv_{\epsilon}+(\mathbb{T}(v_{\epsilon})\cdot n)_{\tau}) |_{\Gamma}=g_{\epsilon}.\label{App_Stokes3}
    \end{align} 
    The   initial condition  is given by 
\begin{align}
(\te(0), v_{\epsilon}(0))=(\theta_{0}, v_{0}). \label{App_ini}
\end{align} 
Note that due to one-way coupling, the flow velocity $v$ does not depend on $\epsilon$, and thus we have 
\begin{align}
v_{\epsilon}=v. \label{v_epsilon}
\end{align}
   However, to distinguish the approximating system from the original one, we still use  the notation $v_{\epsilon}$. 

The outline of the rest of this paper is as follows. We first recall the basic results on Navier slip boundary control for the Stokes problem in Section~\ref{pre}. In Section~\ref{App_sys}, we establish  the convergence of the approximating system governed by \eqref{App_EQ01}--\eqref{App_ini} to the original one governed by  \eqref{EQ01}--\eqref{ini}. Then in Section  \ref{opt_appr_sys}  we show the existence of an optimal solution to the approximating control problem
$(P_{\epsilon})$ and derive the first-order necessary conditions of optimality  by using a variational  inequality. Moreover,  we prove that the optimal solution $(g^{*}_{\epsilon}, v^{*}_{\epsilon}, \theta^*_{\epsilon})$  to the problem~$(P_{\epsilon})$ strongly converges to $(g^{*}, v^{*}, \theta^*)$  as $\epsilon\to 0$, which turns out to be the optimal solution to the original problem $(P)$. Finally, in Section \ref{uniqueness} we prove  that  $(g^{*}, v^{*}, \theta^*)$ is unique  for $d=2$ and $\gamma$ sufficiently large.

 In the sequel, the symbol $C$ denotes a generic positive constant, which is allowed to depend on the domain as well as on indicated parameters.

\section{Preliminary }
\label{pre}
Note that boundary control of the flow velocity  essentially leads to a bilinear control problem for  the scalar equation. As a result of one-way coupling, it is key to  understand the  boundary  control problem of the Stokes equations.  For the convenience of the reader, we recall some  results introduced in \cite{hu2017boundary} on Navier slip boundary control for  Stokes flows. 
In fact, the problems of fluild flows  with Navier slip boundary conditions have been widely studied  in \cite{Beirao2004, clopeau1998vanishing, coron1996controllability,  filho2005inviscid, hu2017partially, hu2018boundary, kelliher2006navier,  lions1969quelques, lions1998mathematical}.

To define the Stokes operator associated with Navier slip boundary conditions, we  introduce 
the  bilinear form  
\begin{align*}
a_0(v, \psi) =2( \mathbb{D}(v),\mathbb{D} (\psi) ) +k\langle v, \psi\rangle,  \quad k>0, \quad  v,\psi \in V^1_{n}(\Omega). 
 \end{align*}
By  Korn's inequality and
trace theorem, it is easy to check that 
$
c_1 \Vert v\Vert^{2}_{H^{1}}\leq a_{0}(v, v)\leq c_2 \Vert v\Vert^{2}_{H^{1}}
$, for some constants $c_1, c_2>0$. Thus  $a_{0}(\cdot, \cdot)$ is $H^1$-coercive. 
Let $(V^{1}_n(\Omega))'$ be the dual space of $V^1_{n}(\Omega)$. Define the operator $A\colon   V^{1}_{n}(\Omega) \rightarrow (V^1_n(\Omega))'$ by 
\begin{align}
  \quad (Av,\psi)=a_{0}(v, \psi). \label{A_a0}
 \end{align}
 The Lax-Milgram Theorem implies  that 
$A\in \mathcal{L}(V^{1}_{n}(\Omega), (V^{1}_n(\Omega))')$.
This  also allows us to identify 
$A$ as an operator acting on $V^{0}_{n}(\Omega)$  with  the domain
 \begin{align*}
 \mathscr{D}(A)&=\{v\in V^{1}_{n}(\Omega)\colon  \psi\mapsto a_0(v, \psi)\ \text{is}\  L^2\text{-continuous} \}.
\end{align*}
In fact, as shown in \cite[(2.9)]{hu2017partially} and  \cite[(5.1)]{kelliher2006navier},  for $v\in V^2_n(\Omega)$  satisfying  the homogenous Navier  slip boundary conditions in \eqref{Stokes3} and  $\psi\in V^1_n(\Omega)$, we have
\begin{align} 
\int_{\Omega} \Delta v\cdot \psi\,dx=-2 \int_\Omega \mathbb{D}(v) \cdot \mathbb{D}(\psi)\, dx 
- \int_{\Gamma}k(v\cdot \tau) ( \psi\cdot \tau)\, dx. 
\label{hh1}
\end{align} 
Thus  \eqref{A_a0}--\eqref{hh1} define the Stokes  operator $ A=-\mathbb{P}\Delta$ with domain 
\[\mathscr{D}(A)=\{v\in V^2_n(\Omega)\colon  kv+(\mathbb{T}(v)\cdot n)_{\tau} |_{\Gamma}=0 \},
\]
 where $\mathbb{P}$ is the Leray projector in $L^2(\Omega)$ on the space $V^0_n(\Omega)$. 
Note that $A$ is self-adjoint, strictly positive, and thus the fractal powers $A^{\sigma}$,  for $ \sigma\in \mathbb{R}$, are well-defined. 
 By  interpolation  theory (cf.~\cite{hu2018boundary, LT2000,  lions1971}), the Navier slip boundary conditions allow us to identify  the  domains of
 $A^{\sigma}$  for $0\leq\sigma\leq 1$ as 
\begin{align}
\mathscr{D}(A^{\sigma})&=V^{2 \sigma}_{n}(\Omega), \quad 0\leq \sigma < \frac{3}{4}, \quad \text{and} \nonumber\\
\mathscr{D}(A^{\sigma})&= \{v\in V^{2\sigma}_{n}(\Omega)\colon kv+(\mathbb{T}(v)\cdot n)_{\tau} |_{\Gamma}=0\},\quad \frac{3}{4} < \sigma \leq 1.\label{1frac_D_Aw}
\end{align}
The detailed proof is given by \cite[Proposition  2.4]{hu2018boundary}. The Navier slip boundary operator $N\colon L^{2}(\Gamma)\to V^{0}_{n}(\Omega)$ is defined by 
 \begin{align*}
 Ng=v \iff  a_0(v, \psi)=\langle g, \psi\rangle, \quad  \psi\in V^{1}_{n}(\Omega). 
 \end{align*}
 Moreover, 
 \begin{align*}
 N\colon L^{2}(\Gamma)\to V^{3/2}_{n}(\Omega)\subset V^{3/2-\delta}_{n}(\Omega) =\mathscr{D}(A^{3/4-\delta/2}),\quad
 \delta>0, 
 \end{align*}
or
 \begin{align*}
 A^{3/4-\delta/2}N\in \mathcal{L}(L^{2}(\Gamma), V^0_{n}(\Omega)), 
 \end{align*}
and 
  \begin{align}
  N^* A\psi=\psi|_{\Gamma}, \quad \psi\in \mathscr{D}(A), \label{adj_N}
\end{align}
where $N^*\colon V^0_n(\Omega)\to L^{2}(\Gamma)$ is the $L^2$-adjoint operator  of $N$ (cf.~\cite{badra2010abstract, hu2017boundary, hu2018boundary, LT2000}).
 By making a change of variable,  we may rewrite the  nonhomogenous boundary problem \eqref{Stokes1}--\eqref{Stokes3} as a  variation of parameters formula 
\begin{align}
   v(t)=e^{-At}v_{0}+(Lg)(t),\label{var_form}
    \end{align}
where  $e^{-At}$ is an analytic semigroup generated by $-A$ on $V^{0}_{n}(\Omega)$ and  $L$ is given by
\begin{align}
(Lg)(t)=\int^{t}_{0} Ae^{-A(t-\tau)}Ng(\tau)\, d\tau.  \label{L_oper}
\end{align}
Recall  the  analytic semigroup theory that  (cf.~\cite[Proposition 0.1]{LT2000}, \cite[p.~74, Theorem 6.13]{pazy1983})
\begin{align}
&e^{-At}\in \mathcal{L}\left(V^0_n(\Omega), L^{2}(0,T; \mathcal{D}(A^{1/2}))\right), \label{semigroup}\\
&\|A^{\sigma}e^{-At} \|\leq Mt^{-\sigma}e^{-\omega t}, \quad \sigma\geq 0, \label{EST_semigroup}
\end{align}
for some $M\geq 1$,  $\omega>0$,  and
\begin{align}
&\int^t_0 e^{A(t-\tau)}\cdot\, d\tau\colon \text{continuous}\ L^2(0, T; V^0_n(\Omega))\to L^2(0, T; \mathcal{D}(A)). \label{convolu}
\end{align}
To understand  the regularity  properties of $L$, we follow the similar approaches as in  (cf.~\cite[Theorem 3.1.4, Theorem 3.1.8]{BLT1}, \cite[Lemmas 3.2.2--3.2.3]{LT2000}, and \cite[Theorems 2.5--2.6]{R12007}) for $v\cdot n|_{\Gamma}=0$ and obtain 
\begin{align}
L\in  \mathcal{L}&\big( L^{2}(0,T; V^{2s}_n(\Gamma)\cap H^{s}(0,T; V^{0}_n(\Gamma)), \nonumber\\
&L^{2}(0,T; V^{2s+3/2}_{n}(\Omega))\cap H^{s+3/4}(0,T; V^0_n(\Omega)\big), \quad  0 \leq s<1/2. \label{5L_oper}
\end{align}
 For $1/2<s\leq 1$, \eqref{5L_oper} holds for  $g(0)=0$.
 This result  is  the same as for classical parabolic problems with Neumann or Robin boundary condition   due to \eqref{adj_N} that $N^*A$ is a Dirichlet trace operator in the case of $v\cdot n|_{\Gamma}=0$. 
 
For $v_0\in V^0_n(\Omega)$, using      \eqref{var_form}   together with \eqref{semigroup}--\eqref{5L_oper} immediately follows  
\begin{align}
\|v\|_{L^{\infty}(0, T;  V^0_n(\Omega))}+\|v\|_{L^{2}(0, T;  V^1_n(\Omega))} +\|\frac{dv}{dt}\|_{L^{2}(0, T; (V^{1}_n(\Omega))')}
\leq C(\|v_{0}\|_{L^2}+\|g\|_{L^2(0,T;V^{ 0}_n(\Gamma)}). \label{EST_v_L2}
\end{align}
      Moreover, if we let $L^*$ be  the $L^2(0, T; \cdot)$-adjoint operator  of $L$, then $L^*$  is given by 
  \begin{align}
(L^{*}\psi)(t)=\int^{T}_{t} N^{*}Ae^{-A(\tau-t)}\psi(\tau)\, d\tau=(\int^{T}_{t} e^{-A(\tau-t)}\psi(\tau)\, d\tau)|_{\Gamma}.
  \label{L_adj}
\end{align}
Slightly modifying the proof in   \cite[Theorem 3.1.9]{BLT1} yields 
\begin{align}
L^{*} \in \mathcal{L}&\big(L^{2}(0,T;V^{2s}_{n}(\Omega))\cap H^{s}(0,T;V^{0}_{n}(\Omega)),  \nonumber\\
 &L^{2}(0,T; V^{2s+3/2}_n(\Gamma))\cap H^{s+3/4}(0,T; V^{0}_{n}(\Gamma))\big), \quad    0\leq s\leq 1.
\label{2L_adj}
    \end{align}
In particular,   letting $s=0$ in \eqref{5L_oper} we have by duality
\begin{align}
L^{*} \in \mathcal{L}&\left(L^{2}(0,T; (V^{3/2}_n(\Omega))'), L^{2}(0,T; V^{0}_n(\Gamma))\right).
\label{3L_adj}
    \end{align}

Next we show the existence of an optimal solution to the problem~$(P)$ for $g\in U_{\text{ad}}=L^{2}(0, T; V^0_n(\Gamma))$.  The proof follows the  standard procedure addressed  in \cite[Theorem~3.2]{hu2017boundary}. First recall the definition of a weak solution to the scalar equation \eqref{EQ01} given by \cite[Definition 3.1]{hu2017boundary}.

\section{Existence of an Optimal Solution to  $(P)$ }
  \begin{definition}
            \label{def1}
         For $\theta_{0}\in L^{\infty}(\Omega)$, $\theta\in C([0, T], (H^1(\Omega))')$  is said to be  a weak solution of equation \eqref{Stokes1}, if $\theta$ satisfies 
  \begin{align}
(  \frac{\partial \theta}{\partial t}, \phi)-(v \theta, \nabla \phi)&=0,  \quad  \forall\phi\in H^{1}(\Omega)\label{1weak_theta},
 \end{align}
where  $v\in L^{2}(0, T; V^{1}_{n}(\Omega))$ satisfies \eqref{var_form} for  $v_{0}\in V^{0}_{n}(\Omega)$ and $g\in V^{0}_{n}(\Gamma)$.
\end{definition}
In fact, according to \cite[Corollary II.1]{lions1998mathematical},  there exists a unique  solution  $\theta\in L^{\infty}(0, T; L^\infty(\Omega))$ to \eqref{EQ01} for $\theta_0\in L^\infty(\Omega)$ and $v\in L^{2}(0, T; V^{1}_{n}(\Omega))$. Moreover, it is straightforward  to check that 
\begin{align}
&\| \frac{\partial \theta}{\partial t}\|^2_{L^{2}(0, T; (H^1(\Omega))') }=\|v\cdot \nabla \theta\|^2_{L^2(0, T; (H^1(\Omega))' )}\nonumber\\
&\qquad=\int^T_0\left(\sup_{\phi\in H^{1}(\Omega)} \frac{|\int_{\Omega}v \cdot \nabla \theta \phi\, dx|}{\|\phi\|_{H^{1}}}\right)^2\, dt\nonumber\\
     &\qquad =\int^T_0\left( \sup_{\phi\in H^{1}(\Omega)} \frac{|\int_{\Omega}v \cdot \nabla (\theta \phi)\, dx-\int_{\Omega}v \theta \cdot \nabla \phi\, dx|}{\|\phi\|_{H^{1}}}\right)^2\, dt\label{0EST_theta_Dt}\\
      &\qquad\leq C\int^T_0\left(\sup_{\phi\in H^{1}(\Omega)} \frac{\|v\|_{L^{2}} \|\theta\|_{L^\infty} \| \phi\|_{H^1}}{\|\phi\|_{H^{1}}}\right)^2\, dt \label{1EST_theta_Dt}\\
       &\qquad\leq C\int^T_0\|v\|^2_{L^{2}} \|\theta\|^2_{L^\infty}\, dt
       =C \|\theta_0\|^2_{\infty}\|v\|^2_{L^2(0, T; L^2(\Omega))}\leq C(\theta_0, v_0, g), \nonumber
\end{align}
where from \eqref{0EST_theta_Dt} to \eqref{1EST_theta_Dt} we used the divergence formula  that 
\begin{align*}
\int_{\Omega}v \cdot \nabla (\theta \phi)\, dx=\int_{\Gamma}(v \cdot n)  \theta \phi\, dx-\int_{\Omega}(\nabla \cdot v) \theta \phi\, dx=0.
\end{align*}
Thus by Aubin-Lions-Simon Lemma (cf.~\cite{boyer2013mathematical}),  we get $\theta\in C([0, T], (H^1(\Omega))')$.

  \begin{thm}\label{existence_opt}
Assume that $\theta_{0}\in L^{\infty}(\Omega)$ and $v_{0}\in V^{0}_n(\Omega)$.  There exists an optimal solution  $g^*\in U_{\text{ad}}$  to the problem~$(P)$.
 \end{thm} 
\begin{proof}
The proof follows the same approach as in \cite[Theorem 3.2]{hu2017boundary}. We provide the complete proof for the convenience of the reader. Since $J$ is bounded from below, we  may choose a minimizing sequence 
 $\{g_{m}\} \subset U_{\text{ad}}$ such that 
 \begin{align}
 \lim_{m\to \infty} J(g_{m})=\inf_{g\in U_{\text{ad}}}J(g). \label{inf_J}
 \end{align}
This also indicates that $\{g_{m}\}$ is uniformly bounded in  $U_{\text{ad}}$, and hence  there exists a weakly convergent subsequence, still denoted by  $\{g_{m}\}$, such that 
    \begin{align}
g_{m}\to g^{*} \quad \text{weakly in}\quad  L^{2}(0, T; V^{0}_{n}(\Gamma)), \ \ \text{as}\ \  m\to \infty.\label{1EST_g}
 \end{align}
With the help of \eqref{EST_v_L2} we can extract  a subsequence $\{v_m\}$ corresponding to $\{g_m\}$, such that 
  \begin{align*}
v_{m}\to v^{*} \quad &\text{weakly  in}\quad L^{2}(0, T; V^{1}_{n}(\Omega)),\\
\frac{\partial v_{m}}{\partial t}\to \frac{\partial v^{*}}{\partial t} \quad & \text{weakly  in}\quad L^{2}(0, T; (V^{1}_{n}(\Omega))'). 
 \end{align*}
 Thus
   \begin{align}
v_{m}\to v^{*} \quad &\text{strongly  in}\quad L^{2}(0, T; V^{0}_{n}(\Omega)) \label{3EST_cong_v}
 \end{align}
(cf.~\cite[Theorem~2.1]{Te}).
Let $\{\theta_m\}$ be the solutions corresponding to  $\{v_{m}\}$ with $\theta_{m}(0)=\theta_0\in L^{\infty}(\Omega)$. Then  $\theta_m\in C([0,T]; (H^1(\Omega))')$ and 
by \eqref{EST_theta_Lp} $\|\theta_{m}(t)\|_{L^{\infty}}=\|\theta_{0}\|_{L^\infty}$ for any $t\geq 0$. 
 Thus there exists   a subsequence,  still denoted by $\{\theta_{m}\}$, satisfying 
  \begin{align}
 \theta_{m}\to  \theta^* \quad \text{weak*  in}\  L^{\infty}(0, T; L^{\infty}(\Omega)). \label{2con_theta}
 \end{align}
Next we  show  that   $\theta^*$  is the solution  corresponding to $v^*$ by Definition \ref{def1}.  Recall  that $v_{m}$ and $\theta_{m}$ satisfy  
  \begin{align}
\left(  \frac{\partial \theta_{m}}{\partial t}, \phi \right)- (v_{m} \theta_{m}, \nabla \phi)&=0,  \quad \phi\in H^{1}(\Omega),\label{2weak_theta}\\
\theta_{m}(0)&=\theta_{0}. \nonumber
  \end{align}
Let $\psi\in C^{1}[0, T]$.  
For each $\phi\in H^1(\Omega)$, 
multiplying  \eqref{2weak_theta}  by $\psi$  and integrating the first term by parts yields  
  \begin{align}
(\theta_m(T),\phi \psi(T))-\int^{T}_{0}(   \theta_{m}, \phi \dot{\psi})\, dt- \int^{T}_{0}(v_{m} \theta_{m}, \nabla\phi \psi)\, dt=(\theta_0,\phi \psi(0)).
\label{3weak_theta}
\end{align}
With the help of  \eqref{2con_theta} and $\phi \dot{\psi}\in L^{1}(0, T; L^1(\Omega))$, it is easy to pass to the limit in the  second term of \eqref{3weak_theta}. Next  we show that  applying \eqref{3EST_cong_v}--\eqref{2con_theta} makes passing to the limit in
the nonlinear term $v_m  \theta_m\to v^*   \theta^*$ possible. 
In fact,  we have
   \begin{align}
& | \int^{T}_{0}\int_{\Omega} (v_{m}\theta_{m})\cdot\nabla (\phi \psi)\, dxdt - \int^{T}_{0}\int_{\Omega}(v^{*}\theta^{*})\cdot\nabla (\phi \psi)\, dxdt|\nonumber\\
 &\quad\leq | \int^{T}_{0}\int_{\Omega} (v_{m}\theta_{m})\cdot\nabla (\phi \psi)- (v^*\theta_{m})\cdot\nabla (\phi \psi)\, dxdt| \nonumber\\
 &\qquad+| \int^{T}_{0}\int_{\Omega} (v^*\theta_{m})\cdot\nabla (\phi \psi)- (v^*\theta^{*})\cdot\nabla (\phi \psi)\, dxdt|
 \nonumber\\
&\quad \leq \int^{T}_{0}\|v_{m}-v^*\|_{L^{2}} \| \theta_{m}\|_{L^{\infty}}\|\nabla \phi\|_{L^2} | \psi|\, dt
+|\int^{T}_{0}\int_{\Omega}(\theta_{m}-\theta^*)v^*\cdot \nabla  (\phi \psi)\, dx\,dt|\nonumber
\end{align}
where
\begin{align}
&\int^{T}_{0}\|v_{m}-v^*\|_{L^{2}} \| \theta_{m}\|_{L^{\infty}}\|\nabla \phi\|_{L^2} | \psi|\, dt \nonumber\\
&\qquad  \leq  \|v_{m}-v^*\|_{L^{2}(0, T; V^{0}_n(\Omega))} \| \theta_{0}\|_{L^{\infty}} \| \nabla \phi\|_{L^{2}}\|\psi\|_{L^{2}(0,T)} \to 0. \label{1EST_vtheta_limit}
\end{align}
Further note that $v^*\cdot \nabla ( \phi\psi)\in L^{1}(0, T; L^{1}(\Omega))$. In light of \eqref{2con_theta} we get 
\begin{align}
|\int^{T}_{0}\int_{\Omega}(\theta_{m}-\theta^*)v^*\cdot \nabla  ( \phi\psi)\, dx\,dt|\to 0.  \label{2EST_vtheta_limit}
  \end{align}  
From  \eqref{1EST_vtheta_limit}--\eqref{2EST_vtheta_limit} we have established   that 
\begin{align} 
v_{m}  \theta_{m}\to v^{*}\theta^*  \quad \text{weakly  in}\quad 
L^{2}(0, T;  (H^{1}(\Omega))'). \label{cong_vtheta}
\end{align}
Furthermore, as proven in  \cite[Theorem~3.2]{hu2017boundary} choosing $\psi\in C^{1}[0, T]$ such that  $\psi(0)=1$ and $\psi(T)=0$, we obtain   $\theta^*(0)=\theta_0$. Similarly, choosing  $\psi\in C^{1}[0, T]$ such that  $\psi(T)=1$  and letting $m\to \infty$ in \eqref{3weak_theta}, we get  
\[\theta_m(T) \to \theta^*(T)\quad \text{weakly in} \quad (H^1(\Omega))'.\]
Clearly, $\theta^*\in C([0, T]; (H^1(\Omega))')$  is the solution  corresponding to $v^*$ based on Definition \ref{def1}.

 Lastly, using the weakly lower semicontinuity property of norms yields 
\[\|g^*\|_{U_{ad}} \leq  \displaystyle \varliminf_{m\to \infty} \|g_m\|_{U_{ad}} \quad \text{and}\quad
\|\theta^*(T)\|_{(H^{1}(\Omega))'} \leq \varliminf_{m\to \infty} \|\theta_m(T)\|_{(H^{1}(\Omega))'}.
\]
In other words,
\[J(g^*)\leq   \varliminf_{m\to \infty} J(g_m)=\inf_{g\in U_{\text{ad}}}J(g),\]
which indicates that $g^*$ is an optimal solution to the problem~$(P)$.

\end{proof}

\section{Convergence of the Approximating System }
\label{App_sys}
 Let $( \theta, v)$ and $( \te, v_{\epsilon})$ be solutions of \eqref{EQ01}--\eqref{ini}  and   \eqref{App_EQ01}--\eqref{App_ini}, respectively, with the same initial condition $(v_0, \theta_{0})$ and boundary condition $g$.  The well-posedness and regularity  of the approximating system follow the results from  parabolic boundary value problems 
(cf.~\cite{brezis2010functional}) and the details can be found in   \cite[Theorem~1]{barbu2016optimal}.  
\begin{thm}
\label{wellposed_te}
For given $\epsilon>0$, $\theta_0\in L^{\infty}(\Omega)$ and  $v_{\epsilon}\in L^2(0, T; V^0_n(\Omega))$, there exists a unique weak solution 
$\te\in C([0, T ]; L^2(\Omega))\cap L^2(0, T;  H^1(\Omega))$ to
\eqref{App_EQ01}, \eqref{App_BC_EQ01} and \eqref{App_ini}. Moreover, 
\begin{align}
&\|\te\|_{L^{\infty}(0, T; L^{\infty}(\Omega))}+\sqrt{\epsilon}\|\te\|_{L^{2}(0, T; H^{1}(\Omega))}+\|\frac{d \te}{dt}\|_{L^{2}(0, T; (H^{1}(\Omega))')}\nonumber\\
&\qquad +\|v_{\epsilon}\cdot \nabla \te\|_{L^{2}(0, T; (H^{1}(\Omega))')}\leq C(\theta_0, v_{\epsilon}). \label{EST_te}
\end{align}
\end{thm}
To show  the convergence of $\theta_\epsilon$ to $\theta$ as $\epsilon\to 0$, we shall need  an {\it a priori} estimate on $\theta$, that is,
 \begin{align}
 \int^T_{0}\|\nabla \theta\|^2_{L^2}\, dt<\infty.\label{EST_thetaH1L2}
 \end{align}
Note  that applying
$H^1$-estimate to scalar equation \eqref{EQ01} and the Gronwall inequality, we obtain  
\begin{align}
 \sup_{t\in[0, T]}\Vert \nabla \theta\Vert_{L^{2}}\leq C \Vert \nabla \theta_{0}\Vert_{L^{2}} e^{ \int^{T}_{0}\|\nabla v\|_{L^{\infty}}\, dt}.
 \label{1EST_te_H1}
\end{align}
The detailed proof can be found in   \cite[Lemma 2.1]{hu2017boundary} and the references therein.  For $\theta_{0}\in H^{1}(\Omega)$, if 
\begin{align}
 \int^{T}_{0}\|\nabla v\|_{L^{\infty}}\, dt <\infty, \label{EST_vH1L1}
\end{align}
then \eqref{1EST_te_H1} holds, and hence  \eqref{EST_thetaH1L2} follows. 
It remains to identify the initial and boundary data of the velocity field such that  \eqref{EST_vH1L1} holds.
Using Agmon  inequality 
(cf.~\cite[(2.21), p.\,11]{temam1995navier})
\begin{align}
\|\nabla v\|_{L^{\infty}} \leq C\|v\|_{H^{1+d/2+\eta}},  \quad d=2,3, \quad  \forall\eta>0,  \label{EST_v_infty}
\end{align}
and the variation of parameters formula 
\begin{align}
   v(t)=e^{-At}(v_{0}-Ng_0)+Ng(t)-\int^{t}_{0} e^{-A(t-\tau)}N\dot{g}(\tau)\, d\tau,  \label{var_form_diff}
    \end{align}
 we have  shown  in \cite{hu2017boundary}  that if 
\[v_0\in V^{d/2-1+2\eta}_{n}(\Omega) \quad \text{and}\quad g\in L^{2}(0, T; V^{d/2-1/2+\eta}_{n}(\Gamma))\cap H^1(0, T; V^0_n(\Gamma)),
 \quad d=2,3,\]
then  \eqref{EST_vH1L1} follows. 
In fact, the first order  time derivative on $g$ can be relaxed, which will be  proven in the following lemma. 

\begin{lemma}\label{App_EST_v}
Let 
\begin{align*}
 \mathbb{S}= L^{2}(0, T; V^{d/2-1/2+2\eta}_{n}(\Gamma))\cap  H^{d/4-1/4+\eta}(0, T; V^{0}_n(\Gamma)),\quad d=2,3,
\end{align*}
equipped with the norm
\begin{align*}
\|g\|_{\mathbb{S}}=\|g\|_{L^{2}(0, T; V^{d/2-1/2+2\eta}_{n}(\Gamma))}
 +\|g\|_{H^{d/4-1/4+\eta}(0, T; V^{0}_n(\Gamma))}, 
 \end{align*}
where $0<\eta<1/4$ for $d=2$ and $0<\eta<1/8$ for $d=3$.  If $v_0\in V^{d/2-1+2\eta}_{n}(\Omega)$   and  $g\in \mathbb{S}$,
then \eqref{EST_vH1L1} holds. 
\end{lemma}

\begin{proof}
We first consider $d=2$ and let $0<\eta<1/4$.  In this case,  $g\in  \mathbb{S}= L^{2}(0, T; V^{1/2+2\eta}_{n}(\Gamma))\cap  H^{1/4+\eta}(0, T; V^{0}_n(\Gamma))$, where $1/4+\eta<1/2$. According to  \eqref{var_form}, \eqref{5L_oper}, and  \eqref{EST_v_infty}, we have
 \begin{align}
&\int^{T}_{0}\|\nabla v\|_{L^{\infty}}\, d\tau 
\leq C \int^{T}_{0}\|v\|_{H^{1+d/2+\eta}}\, d\tau\nonumber\\
&\quad   \leq C\left(\int^T_{0}\|e^{-At}v_{0}\|_{H^{1+d/2+\eta}}\, dt+  \int^T_0\|Lg\|_{H^{1+d/2+\eta}}\, dt\right)\nonumber\\
&\quad \leq C\left( \int^{T}_{0}\|A^{1-\eta/2}e^{-At} A^{-1+\eta/2} A^{1/2+d/4+\eta/2}v_{0}\|_{L^2}\,dt
 +\sqrt{T} \|Lg\|_{L^2(0, T; H^{1+d/2+\eta}(\Omega))}\right) \label{EST_gradV_infty1}\\
&\quad \leq  C\left(\int^T_{0}\frac{e^{-\omega t}}{t^{1-\eta/2}}\,dt \|v_{0}\|_{H^{d/2-1+2\eta}}
+\sqrt{T}\|g\|_{\mathbb{S}}\right) \label{EST_gradV_infty2}\\
&\quad \leq  C \left(\|v_{0}\|_{H^{d/2-1+2\eta}}+ \sqrt{T}\|g\|_{\mathbb{S}}\right).  \label{EST_gradV_infty3}
\end{align}
From \eqref{EST_gradV_infty1}  to \eqref{EST_gradV_infty2} we used  \eqref{EST_semigroup}--\eqref{5L_oper} 
and Young's inequality for convolution. 

 For $d=3$, we shall need $g\in L^{2}(0, T; V^{1+2\eta}_{n}(\Gamma))\cap  H^{1/2+\eta}(0, T; V^{0}_n(\Gamma))$, which indicates that $g\in C([0, T]; V^{2\eta}_n(\Gamma))$, and hence $g(0)$ comes into play in  deriving  the regularity of  $L$ when  $s\geq 1/2$ in \eqref{5L_oper}. In fact, applying integration by parts gives
\begin{align}
   (Lg)(t)=Ng(t)-e^{-At}Ng(0)-\int^{t}_{0} e^{-A(t-\tau)}N\dot{g}(\tau)\, d\tau. \label{Int_L}
   \end{align}
   However, $g(0)$ does not interfere  $Lg\in L^1(0, T; V^{5/2+\eta}_n(\Omega))$.  First of all, it is clear that $Ng\in L^2(0, T; V^{5/2+\eta}_n(\Omega))\subset L^1(0, T; V^{5/2+\eta}_n(\Omega))$ for $T<\infty$. Moreover,  since  $\dot{g}\in L^{2}(0, T; V^{-1+2\eta}_n(\Gamma))$,
we have $N\dot g\in L^{2}(0, T; V^{1/2+2\eta}_n(\Omega))=L^{2}(0, T; D(A^{1/4+\eta}))$ for $0<\eta<1/8$.
Thus  
\[
  \int^{t}_{0} e^{-A(t-\tau)}N\dot{g}(\tau)\, d\tau\in  L^{2}(0, T; D(A^{5/4+\eta/2}))\subset L^{2}(0, T; V^{5/2+\eta}_n(\Omega))).
\]
Lastly, using  the same estimate  as for $v_0$  from \eqref{EST_gradV_infty1} to \eqref{EST_gradV_infty2}, we get for    
$ 0<\eta  <1/8$,
\[
  \int^T_0\| e^{-At}Ng(0)\|_{H^{5/2+\eta}}\, dt \leq C\| Ng(0)\|_{H^{1/2+2\eta}}\leq C\| g(0)\|_{L^2}\leq C\|g\|_{\mathbb{S}}. 
\]
Therefore, \eqref{EST_gradV_infty3} also holds for $d=3$. This completes the proof.
\end{proof}
In the rest of our discussion, $\eta$  always satisfies the assumptions in Lemma \ref{App_EST_v}. The following Theorem establishes the convergence of $\te$ to $\theta$ as $\epsilon\to0$.   To this end, we shall need $\theta_0\in H^1(\Omega)$.

\begin{thm} \label{cog_theta}
Assume that  $\theta_0\in H^{1}(\Omega)$, $v_0\in V^{d/2-1+2\eta}_{n}(\Omega), d=2,3$,  and $g\in \mathbb{S}$. We have
\begin{align}
&\sup_{t\in [0, T]} \| \te-\theta\|_{L^2}\leq  C( \theta_0, v_0, g, T)\epsilon^{1/2} \label{Diff_te}
\end{align}
and 
\begin{align}
\|\frac{\partial \theta_{\epsilon}}{\partial n}|_{\Gamma}- \frac{\partial \theta}{\partial n}|_{\Gamma}\|_{L^2(0, T; H^{-1}(\Gamma))}
\leq  C( \theta_0, v_0, g, T)\epsilon^{1/4}. \label{convg_trace}
\end{align}
\end{thm}
\begin{proof}
Let $\Theta_{\epsilon}=\te-\theta$ and recall $v_{\epsilon}=v$ for given $v_0$ and $g$.
Then based on \eqref{EQ01}--\eqref{ini}  and   \eqref{App_EQ01}--\eqref{App_ini},  $\Theta_{\epsilon}$ satisfies 
\begin{align}
   & \frac{\partial \Theta_{\epsilon}}{\partial t}-\epsilon \Delta \Theta_{\epsilon}+ v\cdot \nabla \Theta_{\epsilon}=\Delta\theta ,  \label{Diff_EQ01}
  \end{align}
with the boundary condition 
   \begin{align}
    \epsilon \frac{\partial\Theta_{\epsilon}}{\partial n}|_{\Gamma}= - \epsilon \frac{\partial \theta}{\partial n}|_{\Gamma}  \label{Diff_EQ01_BC}
      \end{align}
      and  the initial condition
      \begin{align}
      \Theta_{\epsilon}(0)=0. \label{Diff_EQ01_ini}
      \end{align}
  Taking the inner product of \eqref{Diff_EQ01} with $\Theta$ and using \eqref{Diff_EQ01_BC}, we get
  \begin{align}
    &\frac{1}{2}\frac{d \|\Theta_{\epsilon}\|^2_{L^2}}{d t}+\epsilon \|\nabla \Theta_{\epsilon}\|^2_{L^2} 
   =\langle\epsilon \frac{\partial \Theta_{\epsilon}}{\partial n}, \Theta_{\epsilon} \rangle
 +  \langle\epsilon \frac{\partial \theta}{\partial n}, \Theta_{\epsilon}\rangle
   -\epsilon( \nabla \theta, \nabla {\Theta_{\epsilon}}), \nonumber \\
  &\qquad \leq  \epsilon \|\nabla \theta\|_{L^2} \|\nabla \Theta_{\epsilon}\|_{L^2}
 \leq \frac{\epsilon}{2} \|\nabla \theta\|^2_{L^2}+\frac{\epsilon}{2} \|\nabla \Theta_{\epsilon}\|^2_{L^2},  \label{L2_Theta}
  \end{align}
which yields
\begin{align}
    \frac{d \|\Theta_{\epsilon}\|^2_{L^2}}{d t}  +   \epsilon \|\nabla \Theta_{\epsilon}\|^2_{L^2} 
\leq \epsilon \|  \nabla \theta\|^2_{L^{2}}. \label{L2_Theta_aa}
  \end{align}
By the Gronwall inequality and the  initial condition \eqref{Diff_EQ01_ini}, we have
\begin{align}
    \|\Theta_{\epsilon}\|^2_{L^2}
     &\leq  \epsilon \int^t_{0}e^{-\epsilon(t-\tau)}\|\nabla   \theta\|^2_{L^{2}}\,d\tau
     \leq  \epsilon \int^T_{0}\| \nabla  \theta\|^2_{L^{2}}\,d\tau\nonumber\\
     &\leq  \epsilon  T\sup_{t\in [0, T]}\| \nabla  \theta\|^2_{L^{2}}\leq  C(\theta_0, v_0, g, T) \epsilon,  \label{2L2_Theta}
  \end{align}
where we used \eqref{1EST_te_H1} and Lemma \ref{App_EST_v} for deriving  the last inequality. 
Moreover, from \eqref{L2_Theta_aa},
\begin{align}
\int^T_{0} \|\nabla \Theta_{\epsilon}\|_{L^2}\, d\tau\leq  \int^T_{0}\| \nabla  \theta\|^2_{L^{2}}\,d\tau\leq C(\theta_0, v_0, g, T).
\label{H1_Theta}
\end{align}

To establish \eqref{convg_trace}, using  the trace theorem together with \eqref{2L2_Theta}--\eqref{H1_Theta} we obtain 
\begin{align*}
& \|\frac{\partial \Theta_{\epsilon}}{\partial n}|_{\Gamma}\|_{L^2(0, T; H^{-1}(\Gamma))}
\leq C \|\Theta_{\epsilon}\|_{L^2(0, T; H^{1/2}(\Omega))}\\
&\qquad \leq  C( \int^T_0\|\Theta_{\epsilon}\|_{L^{2}} \|\Theta_{\epsilon}\|_{H^{1}}\, dt)^{1/2}\\
&\qquad\leq C(\sup_{t\in[0, T]} \|\Theta_{\epsilon}\|_{L^{2}})^{1/2}  (\int^T_0\|\Theta_{\epsilon}\|_{H^{1}}\, dt)^{1/2}
\leq  C( \theta_0, v_0, g, T)\epsilon^{1/4}.
\end{align*}
This completes the proof. 
\end{proof}

Note that since there is no boundary condition imposed on  $\theta$, $\frac{\partial \theta}{\partial n}|_{\Gamma}$ is not defined.
In addition, under the assumptions of Theorem \ref{cog_theta} we can further verify that 
$ \frac{\partial \theta}{\partial t}\in L^{2}(0, T; L^2(\Omega)) $.   Again by  the Aubin--Lions--Simon Lemma,  we have 
\begin{align*}
\theta\in   L^{\infty}(0, T; H^1(\Omega))\cap  H^1(0, T; L^2(\Omega))\subset  C([0, T]; H^{1/2}(\Omega)). 
\end{align*}
Combining this with Theorem \ref{wellposed_te} and \eqref{Diff_te} yields  
\begin{align}
 \| \te(T)-\theta(T)\|_{L^2}\leq  C( \theta_0, v_0, g, T)\epsilon^{1/2}. \label{Diff_teT}
\end{align}
\section{Existence of an Optimal Solution to  $(P_{\epsilon})$ and its  Conditions of Optimality}
\label{opt_appr_sys}
Note that the existence of an optimal controller to  the problem~$(P)$ is independent of $\epsilon$. 
With the help of  \eqref{v_epsilon} and  \eqref{EST_te}, the existence of  an optimal controller to the problem~$(P_{\epsilon})$ follows immediately.

  \begin{thm}\label{App_existence_opt}
 Assume that $\theta_{0}\in L^{\infty}(\Omega)$ and $v_{0}\in V^{0}_n(\Omega)$.  There exists an optimal solution  $g^*_{\epsilon}\in U_{\epsilon_\text{ad}}$  to the problem~$(P_{\epsilon})$.
 \end{thm} 

We now derive the first-order  necessary  optimality conditions  for the problem`$(P_{\epsilon})$  by using a variational inequality  (cf.~\cite{lions1971}), that is,  
if $g_{\epsilon}$ is an optimal solution of the problem~$(P_{\epsilon})$, then
\begin{align} 
J'_{\epsilon}(g_{\epsilon})\cdot (f_{\epsilon}-g_{\epsilon})\geq 0, \quad  f_{\epsilon}\in U_{\epsilon_{ad}}. \label{App_var_ineq}
\end{align}

Let $w_{\epsilon}=v'(g_{\epsilon})\cdot h_{\epsilon}$ be the G\^{a}teaux derivative of $v_{\epsilon}$ with respect to $g_{\epsilon}$ in every direction $h$ in 
 $U_{\epsilon_{ad}}$.  Then  by \eqref{5L_oper}, we have  
 \begin{align}
 w_{\epsilon}&=(Lh_{\epsilon})(t)  \in 
 L^{2}(0,T; V^{3/2}_{n}(\Omega))\cap H^{3/4}(0,T; V^0_n(\Omega))\subset C([0,T]; V^{1/2}_n(\Omega)), \label{App_EST_w} 
 \end{align}
which  is the solution to the  Stokes equations \eqref{Stokes1}--\eqref{ini} with the boundary condition $g=h_{\epsilon}$ and the initial condition is zero.

Now denote by  $\tz=\te'(g)\cdot h_{\epsilon}$ the G\^{a}teaux derivative of  $\te$ with respect to $g_{\epsilon}$. Then $\tz$ satisfies the  equation
  \begin{align}
  \frac{ \partial \tz}{\partial t} -\epsilon \Delta \tz+ w_{\epsilon}\cdot \nabla \te+v_{\epsilon}\cdot \nabla \tz=0,
    \label{linearized_te}
    \end{align}
    with  the boundary condition 
    \begin{align}
    \epsilon\frac{\partial{\tz}}{\partial{n}}|_{\Gamma}=0\label{BC_tz}
    \end{align}
    and the initial condition 
    \begin{align}
  \tz(0)=0.\label{IC_tz}
  \end{align}
  To show that  \eqref{linearized_te}--\eqref{IC_tz} is well-posed, we  first establish an {\it a priori} estimate for~$\tz$.
  Taking the inner product of \eqref{linearized_te} with $\tz$ gives 
    \begin{align*}
&  \frac{1}{2}\frac{ d\|\tz\|^{2}_{L^{2}}}{d t}+\epsilon\|\nabla \tz\|^2_{L^2} =
-\int_{\Omega}(w_{\epsilon}\cdot \nabla \te) \tz\, dx
  -\int_{\Omega} (v_{\epsilon}\cdot \nabla \tz)\tz\, dx\\
  &\qquad=-\int_{\Omega}w_{\epsilon}\cdot \nabla ( \te \tz)\, dx
  +\int_{\Omega}(w_{\epsilon}  \te) \cdot \nabla  \tz\, dx
  -1/2\int_{\Omega} v_{\epsilon}\cdot \nabla \tz^2\, dx\\
  &\qquad= \int_{\Omega}(w_{\epsilon}  \te)\cdot \nabla \tz\, dx
   \leq \|w_{\epsilon}\|_{L^{4}}\| \te\|_{L^{4}}\|\nabla \tz\|_{L^{2}}\\
     &\qquad  \leq C \|w_{\epsilon}\|_{H^{d/4}}\| \te\|_{H^{d/4}}\|\nabla \tz\|_{L^{2}}, \quad d=2,3,\\
     &\qquad \leq C\|w_{\epsilon}\|^2_{H^{d/4}}\| \te\|^2_{H^{d/4}}
 +\frac{\epsilon}{2}\|\nabla \tz\|^2_{L^{2}},
\end{align*}
which follows
         \begin{align}
 \frac{ d\|\tz\|^{2}_{L^{2}}}{d t}+\epsilon\|\nabla \tz\|^2_{L^2}  
  \leq C\|w_{\epsilon}\|^2_{H^{d/4}}\| \te\|^2_{H^{d/4}}.\label{EST_tz}
\end{align}
To complete the estimate, it suffices to show the right hand side of \eqref{EST_tz} is integrable. Note that  
	 \begin{align*}
	&\int^T_{0}\|w_{\epsilon}\|^2_{H^{d/4}}\| \te\|^2_{H^{d/4}}\, dt\\
	&\quad  \leq  \begin{cases}
  C\int^T_0 \|w_{\epsilon}\|_{L^{2}}\|\nabla w\|_{L^{2}}\| \te\|_{L^{2}}\|\nabla \te\|_{L^{2}}\, dt
  \quad \text{if } \ d=2,\\
  C\int^T_0\|w_{\epsilon}\|^{3/2}_{H^{1/2}} \|w_{\epsilon}\|^{1/2}_{H^{3/2}}\| \te\|^{1/2}_{L^{2}}\|\nabla \theta_{\epsilon}\|^{3/2}_{L^2}\, dt
  \quad \text{if } \ d=3.
     \end{cases}\\
&\quad\leq
  \begin{cases}
    C\| w_{\epsilon}\|_{L^{\infty}(0, T; L^{2}(\Omega))}\| \te\|_{L^{\infty}(0, T; L^2(\Omega))} 
 \|w_{\epsilon}\|_{L^2(0, T; H^{1}(\Omega))}	 \|\te\|_{L^2(0, T; H^{1}(\Omega))} \quad \text{if } \ d=2,\\
   C\| w_{\epsilon}\|^{3/2}_{L^{\infty}(0, T; H^{1/2}(\Omega))}\| \te\|^{1/2}_{L^{\infty}(0, T; L^2(\Omega))} 
	 \| w_{\epsilon}\|^{1/2}_{L^2(0, T; H^{3/2}(\Omega))} \| \te\|^{3/2}_{L^2(0, T; H^{1}(\Omega))} \quad \text{if } \ d=3.
  \end{cases}
	 	\end{align*}
Applying   \eqref{EST_te} and \eqref{App_EST_w} gives 
	  \begin{align*}
&\|\tz\|^{2}_{L^{2}}+\epsilon\int^t_0\|\nabla \tz\|^2_{L^2}\, dt
< \infty.
 \end{align*}
The rest of the proof follows the standard approaches  for parabolic problems
 (cf.~\cite[p.\,342]{brezis2010functional}).

In order to   apply   the variational inequaity  \eqref{App_var_ineq} to derive   the optimality system,
we first rewrite  the  cost functional $J_{\epsilon}$  as 
 \begin{align*}
 J_{\epsilon}(g_{\epsilon})=\frac{1}{2}(\Phi_{\epsilon}(T), \te(T))
 +\frac{\gamma}{2}\int^T_{0}\langle g_{\epsilon}, g_{\epsilon}\rangle\, dt, \quad (P'_{\epsilon})
 \end{align*}
where $\Phi_{\epsilon}$ satisfies 
\begin{align}
\mathcal{A}\Phi_{\epsilon}(T)&=\theta_{\epsilon}(T) \label{Phi1} \\
 \frac{\partial \Phi_{\epsilon}(T)}{\partial n}&=0. \label{Phi2}
\end{align}
The Neumann boundary value problem \eqref{Phi1}--\eqref{Phi2} has a unique solution  $\Phi_{\epsilon}(T)=\mathcal{A}^{-1}\te(T)$
(cf.~\cite{lions1969quelques}, \cite{umezu1994lp})
and $\Phi_{\epsilon}(T)\in H^2(\Omega)$ due to $\te(T)\in L^2(\Omega)$ by  Theorem \ref{wellposed_te}.
 The   variational inequality \eqref{App_var_ineq} becomes
  \begin{align}
 J'_{\epsilon}(g_{\epsilon})\cdot h_{\epsilon}=(\Phi_{\epsilon}(T), \tz(T))
 +\gamma \int^T_{0}\langle g_{\epsilon}, h_{\epsilon}\rangle\, dt\geq  0, \quad h_{\epsilon}\in U_{\epsilon_{ad}}.  \label{app_varcost}
 \end{align}
 For given $\epsilon>0$,  the adjoint system  associated  with the cost functional $(P'_{\epsilon})$ is defined by	
\begin{align}
-\frac{ \partial{\rho_{\epsilon}}}{\partial t}- \epsilon \Delta \rho_{\epsilon}-v_{\epsilon}\cdot \nabla   \rho_{\epsilon}=0, \quad\label{App_rho}
\end{align}
with the boundary condition
\begin{align}
&\epsilon\frac{\partial \rho_{\epsilon}}{\partial n}|_{\Gamma}=0	 \label{BC_App_rho}
\end{align}
and the final time  condition
\begin{align}
\rho_{\epsilon}(T)=\Phi_{\epsilon}(T)\in H^2(\Omega), \label{App_rhoT}
\end{align}
where $v_{\epsilon}=v$  satisfies  the  Stokes equations \eqref{App_Stokes1}--\eqref{App_Stokes2} and \eqref{App_Stokes3}--\eqref{App_ini}.
Since $\rho_{\epsilon}(T)\in H^2(\Omega)$,  the compatibility condition for final and boundary data need to be satisfied, i.e., 
$\epsilon\frac{\partial \rho_{\epsilon}(T)}{\partial n}|_{\Gamma}=0$. This is  indeed true by \eqref{Phi2}. However, the compatibility condition will not get in the way as $\epsilon\to0$.

Replacing  $t$ by $T-t$  and  using the similar approach  as in  the proof of Theorem \ref{wellposed_te}, we obtain  that  there exists a unique solution $\rho_{\epsilon}\in C([0, T]; H^1(\Omega))\cap L^2(0, T; H^2(\Omega))$ to \eqref{App_rho}--\eqref{App_rhoT}, which  satisfies 
\begin{align}
&\|\tr\|_{L^{\infty}(0, T; H^{1}(\Omega))}+\sqrt{\epsilon}\|\tr\|_{L^{2}(0, T; H^{2}(\Omega))}+\|\frac{d \tr}{dt}\|_{L^{2}(0, T; L^{2}(\Omega))}\nonumber\\
&\qquad +\|v_{\epsilon}\cdot \nabla \tr\|_{L^{2}(0, T; L^{2}(\Omega))}\leq C(\te(T), v_{\epsilon}). \label{EST_tr}
\end{align}
If $v_{\epsilon}$ further satisfies  \eqref{EST_vH1L1},  then  using the same argument as in the proof of  Theorem \ref{cog_theta}  and   the relation between  the final conditions  given by \eqref{Diff_teT}, we have   
\begin{align}
\sup_{t\in [0, T]} \| \tr-\rho\|_{L^2}\to 0, \ \ \text{as}\ \ \epsilon\to 0, \label{covg_rhoL2}
 \end{align} 
 where $\rho$ satisfies 
 \begin{align}
 &- \frac{ \partial}{\partial t} \rho 
  -v\cdot \nabla \rho=0,  \quad v=v_{\epsilon}, \label{adj}\\
  &\rho(T)=\Lambda^{-2}\theta(T),
    \label{adj_final}
  \end{align}
  and  
  \begin{align}
  \sup_{t\in[0, T]}\|\nabla \rho\|_{L^2}<\infty. \label{EST_rho_H1}
  \end{align}
 In fact,  \eqref{adj}--\eqref{adj_final} define the adjoint system  of \eqref{EQ01}--\eqref{ini} associated with the cost functional $J$.  
We now establish the optimality system of the approximating problem $(P_{\epsilon})$ and its convergence  to the optimality  system of the problem~$(P)$.

\begin{thm}
\label{thm_App_opt_cond}
 Let $\theta_{0}\in L^{\infty}(\Omega)$ and $v_{0}\in V^{0}_{n}(\Omega)$. 
Assume that $g^{*}_{\epsilon}$ is an optimal controller of the problem~$(P_{\epsilon})$. If $(v_{\epsilon}, \theta_{\epsilon})$ is the corresponding  solution of  \eqref{App_EQ01}--\eqref{App_ini} and $\rho_{\epsilon}$ is the solution of the  adjoint equations \eqref{App_rho}--\eqref{App_rhoT} associated with $(v_{\epsilon}, \theta_{\epsilon})$
then 
\begin{align}
g^{*}_{\epsilon}=-\frac{1}{\gamma}L^{*}(\mathbb{P}(\theta_{\epsilon} \nabla \rho_{\epsilon}))
\in L^{2}(0,T; V^{3/2}_n(\Gamma))\cap H^{3/4}(0,T; V^0_n(\Gamma)). \label{App_opt_cond}
\end{align}
  \end{thm}
\begin{proof}
First 
multiplying (\ref{linearized_te}) by $\rho_{\epsilon}$, we have
   \begin{align*}
&\int^{T}_{0}\left(  \frac{ \partial}{\partial t} \tz, \rho_{\epsilon}\right)\,dt
  + \int^{T}_{0}((Lh_{\epsilon}) \cdot \nabla \te, \rho_{\epsilon})\,dt
  +\int^{T}_{0}(v_{\epsilon}\cdot \nabla \tz, \rho_{\epsilon})\,dt
  =\int^{T}_{0}(\Delta \tz, \rho_{\epsilon})\, dt. 
  \end{align*}
  Integrating the first term with respect to $t$ and the third term with respect to $x$  yield
  \begin{align*}
& -\int^T_{0}\left( \frac{ \partial}{\partial t}\rho_{\epsilon}, \tz\right)\, dt + (\rho_{\epsilon}(T),\tz(T))
+ \int^{T}_{0}((Lh_{\epsilon}))\cdot \nabla \te, \rho_{\epsilon})\,dt\\
 &\qquad-\int^{T}_{0}(\tz, v_{\epsilon}\cdot \nabla \rho_{\epsilon})\,dt=\int^{T}_{0}(\epsilon  \tz, \Delta\rho_{\epsilon})\, dt,
  \end{align*}
where we used $(\Delta \tz, \rho_{\epsilon})=(\epsilon  \tz, \Delta\rho_{\epsilon})$ due to \eqref{BC_tz} and  \eqref{BC_App_rho}.
 In light of the adjoint equation   \eqref{App_rho} and the final condition  \eqref{App_rhoT}, we have
   \begin{align}
   (\Phi_{\epsilon}(T), z_{\epsilon}(T))= (\rho_{\epsilon}(T), z_{\epsilon}(T))
   = -\int^{T}_{0}((Lh_{\epsilon})\cdot \nabla \theta_{\epsilon}, \rho_{\epsilon})\,dt.
    \label{App_2adj}
  \end{align}
Combining  \eqref{app_varcost} with \eqref{App_2adj} yields 
 \begin{align}
 J'_{\epsilon}(g_{\epsilon})\cdot h_{\epsilon}=& -\int^{T}_{0}((Lh_{\epsilon})\cdot \nabla \theta_{\epsilon}, \rho_{\epsilon})\,dt
 +\gamma \int^T_{0}\langle g_{\epsilon}, h_{\epsilon}\rangle\, dt\geq 0.\label{2app_varcost}
 \end{align}
Note that   $\nabla\cdot (Lh_{\epsilon})=0$ and  $(Lh_{\epsilon})\cdot n|_{\Gamma}=0$. Thus
\begin{align}
\int^{T}_{0}((Lh_{\epsilon}) \cdot \nabla \te, \rho_{\epsilon})\,dt&=\int^{T}_{0}\int_{\Omega}(\mathbb{P} (Lh_{\epsilon})) \cdot \nabla (\te\rho_{\epsilon})\,dx\,dt-\int^{T}_{0}(\mathbb{P} (Lh_{\epsilon}),  \te \nabla \rho_{\epsilon})\,dt\label{0App_EST_Lh}\\
&=-\int^{T}_{0}(h_{\epsilon},  L^{*}\mathbb{P}(\te \nabla \rho_{\epsilon}))\,dt, \label{App_EST_Lh}
\end{align}
where $\te \nabla \rho_{\epsilon}\in L^2(0, T; L^2(\Omega))$. Therefore, based on  \eqref{2app_varcost}--\eqref{App_EST_Lh}  if $g^{*}_{\epsilon}$ is an optimal solution, then 
 \begin{align*}
 J'_{\epsilon}(g^*_{\epsilon})\cdot h_{\epsilon}=&\int^{T}_{0}(h_{\epsilon},  L^{*}(\mathbb{P}(\te \nabla \rho_{\epsilon})))\,dt
 +\gamma \int^T_{0}\langle g_{\epsilon}, h_{\epsilon}\rangle\, dt\geq 0, \quad h_{\epsilon}\in U_{\epsilon_{ad}}, 
 \end{align*}
 which gives
\[g^{*}_{\epsilon}=-\frac{1}{\gamma}L^{*}(\mathbb{P}(\te \nabla \rho_{\epsilon})). \]
Moreover, by the continuity of $\mathbb{P}$ on $L^{2}(\Omega)$ (cf.~\cite[p.\,13]{Te}), \eqref{EST_te}, and \eqref{EST_tr}, we have 
\begin{align}
&\|\mathbb{P}(\te \nabla \rho_{\epsilon}))\|_{ L^{2}}
\leq C \|\te \nabla \rho_{\epsilon}\|_{L^{2}(0, T:  L^{2}(\Omega))} \nonumber\\
&\qquad\leq C\|\te\|_{L^{\infty}(0, T; L^{\infty}(\Omega))} \|\tr\|_{L^{2}(0, T; H^{1}(\Omega))}<\infty. \label{EST_tetr}
\end{align}
Lastly,  combining \eqref{EST_tetr} with the regularity property of $L^*$ given by \eqref{2L_adj} yields  \eqref{App_opt_cond}. This completes the proof.
\end{proof}

\begin{remark}
\label{projector}
As mentioned in \cite[Remark 6]{badra2010abstract},  since the Leray projector $\mathbb{P}\colon L^2(\Omega)\to V^0_n(\Omega)$ can be extended from $H^s(\Omega), s>0$, to $V^s_n(\Omega)$,   its adjoint $\mathbb{P}^*\colon V^0_n(\Omega)\to L^2(\Omega)$ can be extended as a bounded
operator from $(V^s_n(\Omega))'$ to $(H^s(\Omega))'$ by 
\begin{align}
(\mathbb{P}^*\psi,  \varphi)_{(H^s(\Omega))', H^s(\Omega)}=(\psi,  \mathbb{P} \varphi)_{(V^s_n(\Omega))', V^s_n(\Omega)},
\quad \psi\in(V^s_n(\Omega))',\  \varphi\in V^s_n(\Omega). \label{duality_P}
\end{align}
Therefore, if $\te \nabla \rho_{\epsilon}\in L^2(0, T; H^s(\Omega))$, where $s<0$,  then we  use duality from \eqref{0App_EST_Lh} to \eqref{App_EST_Lh}  and replace $\mathbb{P}$ by $\mathbb{P}^*$. In this case,
\begin{align}
g^{*}_{\epsilon}=-\frac{1}{\gamma}L^{*}(\mathbb{P}^*(\te \nabla \rho_{\epsilon})). \label{opt_dualform}
\end{align}
\end{remark}
%

 In order to address the convergence of the  optimality conditions for the approximating problem,  we shall assume $\theta_0\in L^{\infty}(\Omega)\cap H^1(\Omega)$ and $v_0\in V^{d/2-1+2\eta}_n(\Omega)$, $d=2,3$, in the rest of our discussion.
\begin{thm}
\label{cong_opt_pair}
Assume $\theta_0\in L^{\infty}(\Omega)\cap H^1(\Omega)$ and $v_0\in V^{d/2-1+2\eta}_n(\Omega)$, $d=2,3$. Let $(g^*_{\epsilon}, v^*_{\epsilon},  \te^*)$ be an optimal solution for $(P_{\epsilon})$. Then, there exists  $(g^*, v^*, \theta^*)$ such that
 a subsequence of $(g^*_{\epsilon}, v^*_{\epsilon},  \te^*)$ in terms of $\epsilon$, still denoted by  $\{g^*_{\epsilon}, v^*_{\epsilon},  \te^*\}$, satisfying 
\begin{align*}
&g^{*}_{\epsilon}\to g^* \  \ \text{strongly in} \ \   L^{2}(0, T; V^{3/2-\delta}_{n}(\Gamma)),\\
&v^{*}_{\epsilon}\to v^* \  \ \text{strongly in} \ \   L^{2}(0, T; V^{d/2+2\eta}_{n}(\Omega)), 
\end{align*}
for   $0<\delta\leq  3-d/2-2\eta$, and 
\begin{align}
\te^{*}\to \theta^* \ \ \text{strongly in}  \ \ L^{2}(\Omega), \ \text{uniformly in} \ t\in [0, T],  \ \label{cong_thetastar}
\end{align}
as $\epsilon \to 0$.
Moreover, $(g^*, v^*, \theta^*)$ is an optimal solution to the problem~$(P)$, which can be solved from 
\begin{align}
g^{*}=-\frac{1}{\gamma}L^{*}(\mathbb{P}(\theta^* \nabla \rho^*))
\in L^{2}(0, T; V^{3/2}_{n}(\Gamma))\cap H^{3/4}(0, T; V^{0}_{n}(\Gamma)),
\label{Ori_pt_cond}
 \end{align}
 where $\rho^*$ is the solution to the dual problem \eqref{adj}--\eqref{adj_final} corresponding to $(v^*, \theta^*)$.

\end{thm}
\begin{proof} Step 1: We first show the strong convergence of  the optimal solution  to the problem~$(P_{\epsilon})$. 
Wth the help of  \eqref{EST_te} and  \eqref{EST_tr}   we get
\begin{align}
    \|\theta^*_{\epsilon}\nabla \rho^*_{\epsilon}\|_{L^2(0, T; L^2(\Omega))}
   \leq \|\theta^*_{\epsilon}\|_{L^{\infty}(0, T; L^\infty(\Omega))}\|\nabla \rho^*_{\epsilon}\|_{L^2(0, T; L^2(\Omega))}\leq C,  \label{convg_rhote} 
 \end{align}
independent of $ \epsilon$. Thus there exist  subsequences, still denoted by $\{\theta^*_{\epsilon}\nabla \rho^*_{\epsilon}\}$ 
and $\{ \nabla  \rho^*_{\epsilon} \}$, such   that  
 \begin{align*}    
 \theta^*_{\epsilon}\nabla \rho^*_{\epsilon}\to \xi \ \ \text{weakly in }\ \ L^2(0, T; L^2(\Omega))  \ \ \text{as}\ \  \epsilon\to 0, 
\end{align*}
for some  $\xi \in L^2(0, T; L^2(\Omega))$,
and 
\begin{align}
& \nabla  \rho^*_{\epsilon}   \to \nabla \rho^*   
    \ \ \text{weakly in }\ \ L^2(0, T; L^2(\Omega))  \ \ \text{as}\ \  \epsilon\to 0.
\label{convg_rhoH1}
   \end{align}
Based on the property  of $L^*$ given by \eqref{5L_oper} and the continuity of $\mathbb{P}$, there exists a subsequence $\{g^*_{\epsilon}=L^*( \mathbb{P}(\theta^*_{\epsilon}\nabla \rho^*_{\epsilon}))\}$ in terms of $\epsilon$, such that 
\begin{align*}
g^*_{\epsilon}\to g^*=L^*(\mathbb{P}\xi) \quad  &\text{weakly in}\quad  L^{2}(0, T; V^{3/2}_{n}(\Gamma))\cap H^{3/4}(0, T; V^{0}_n(\Gamma)), \ \ \text{as}\ \ \epsilon\to 0, 
\end{align*}
Thus in light of \cite[Theorem~2.2, p.\,186]{Te}, 
\begin{align*}
g^*_{\epsilon}\to g^* \quad \text{strongly in}\quad  L^{2}(0, T; V^{3/2-\delta}_{n}(\Gamma)), \quad\forall 0<\delta<\frac{3}{2}. 
\end{align*}
Correspondingly,  by \eqref{var_form}, \eqref{semigroup}, and  \eqref{5L_oper}, we have for $v_0\in V^{d/2-1+2\eta}_n(\Omega), d=2,3,$ that 
  \begin{align}
v^*_{\epsilon}-v^{*}=L(g^*_{\epsilon}-g^*)\to 0\quad \text{strongly in}\quad &L^{2}(0, T; V^{\min\{d/2+2\eta, 3-\delta\}}_{n}(\Omega))\nonumber\\
&=L^{2}(0, T; V^{d/2+2\eta}_{n}(\Omega)).
\label{App_1EST_vstar}
 \end{align}
 for $0<\delta \leq 3-d/2-2\eta$. Let $\theta^*$ be the solution of  \eqref{EQ01}  associated with $v^*$ and initial condition $\theta_0$.  Next we prove
 \begin{align}
 \sup_{t\in [0, T]}\|\nabla\theta^*\|_{L^2}<\infty.\label{App_EST_thetastar_H1}
 \end{align}
According to Lemma \ref{App_EST_v}, we know that  $\int^{T}_{0}\|\nabla v^*\|_{L^{\infty}}\, dt<\infty$ for $v_0\in V^{d/2-1+2\eta}_n(\Omega)$, $d=2,3$, and 
$g^*\in L^{2}(0,T; V^{3/2}_n(\Gamma))\cap H^{3/4}(0,T; V^0_n(\Gamma))$.
Thus \eqref{App_EST_thetastar_H1} follows immediately from \eqref{1EST_te_H1}.

To establish  \eqref{cong_thetastar}, we let $\Theta^*_{\epsilon}=\te^*-\theta^*$ and $V^*_{\epsilon}=v^*_{\epsilon}-v^*$. Then    
 $\Theta^*_{\epsilon}$ satisfies 
\begin{align}
   & \frac{\partial \Theta^*_{\epsilon}}{\partial t}-\epsilon \Delta \Theta^*_{\epsilon}+ v^*_{\epsilon}\cdot \nabla \Theta^*_{\epsilon}
   +V^*_{\epsilon}\cdot \nabla \theta^*
   =\Delta\theta^* ,  \label{Diff_EQ01star}
  \end{align}
  with the boundary condition
   \begin{align}
    \epsilon \frac{\partial \Theta^*_{\epsilon}}{\partial n}|_{\Gamma}= - \epsilon \frac{\partial \theta^*}{\partial n} |_{\Gamma}\label{App_BC_EQ01star}
      \end{align}
      and the initial condition $\Theta^*_{\epsilon}(0)=V^*_{\epsilon}(0)=0$.
Recall by \eqref{App_1EST_vstar} that 
   \begin{align}
V^*_{\epsilon}\to 0\quad \text{strongly in}\ \  L^{2}(0, T; V^{d/2+2\eta}_{n}(\Omega)), \ \ \text{as}\ \ \epsilon\to 0.\label{App_1EST_Vstar}
 \end{align}
As shown   in  the proof of Theorem~\ref{cog_theta}, applying  $L^2$-estimate for $\Theta^*_{\epsilon}$ follows 
\begin{align}
    \frac{d \|\Theta^*_{\epsilon}\|^2_{L^2}}{d t}+\epsilon \|\nabla \Theta^*_{\epsilon}\|^2_{L^2} 
     \leq &\epsilon \| \nabla  \theta^*\|^2_{L^{2}}+\|V^*_{\epsilon}\cdot \nabla \theta^*\|^2_{L^2}+\|\Theta^*_{\epsilon}\|^2_{L^2}. \label{L2_ThetaStar}
  \end{align}
Using  the Gronwall inequality and $\Theta^*_{\epsilon}(0)=0$, we get
  \begin{align}
    \|\Theta^*_{\epsilon}\|^2_{L^2}
     &\leq \int^t_{0}e^{(t-\tau)}(\epsilon \| \nabla  \theta^*\|^2_{L^{2}}+\|V^*_{\epsilon}\cdot \nabla \theta^*\|^2_{L^2})\,d\tau, \label{2L2_ThetaStar}
  \end{align} 
  where  by \eqref{App_EST_thetastar_H1} and \eqref{App_1EST_Vstar}, we have 
  \begin{align*}
  &\sup_{t\in[0, T]}  \|\Theta^*_{\epsilon}\|^2_{L^2}\leq \epsilon (e^T-1)\sup_{t\in[0, T]}\|\nabla \theta^*\|^2_{L^2}
  +e^T\int^T_{0}\|V^*_{\epsilon}\|^2_{L^\infty}\| \nabla \theta^*\|^2_{L^2}\,d\tau\\
 & \qquad \leq   \epsilon (e^T-1)\sup_{t\in[0, T]}\|\nabla \theta^*\|^2_{L^2}+e^T(\sup_{t\in [0, T]} \|\nabla \theta^*\|^2_{L^2}) \|V^*_{\epsilon}\|^2_{L^2(0, T; V^{d/2+\eta}_n(\Omega))}\to 0,
 \end{align*}
 as $\epsilon \to 0$. Therefore, \eqref{cong_thetastar} holds.

Step 2: We claim that $(g^*, v^*, \theta^*) $ is an optimal solution  to the problem~$(P)$.  Since $(g^*_{\epsilon}, v^*_{\epsilon}, \theta^*_{\epsilon})$ is an optimal solution to  the problem~$(P_{\epsilon})$, we have
\begin{align*}
\|\Lambda^{-1}\theta^*_{\epsilon}(T)\|^2_{L^2} +\int^T_{0}\|g^*_{\epsilon}\|^2_{L^2(\Gamma)}\, dt
\leq\|\Lambda^{-1}\theta_{\epsilon}(T)\|^2_{L^2} +\int^T_{0}\|g\|^2_{L^2(\Gamma)}\, dt,
\end{align*}
for any $g\in L^{2}(0, T; V^0_n(\Gamma))$, where $\te$ is the solution of \eqref{App_EQ01} associated with $(g, v_{\epsilon})$. Letting $\epsilon\to 0$ and  using the weakly lower semicontinuity  of norms, the continuity of $\Lambda^{-1}$, and the strong convergence of $\theta^*_{\epsilon}$ to $\theta^*$ and $\theta_{\epsilon}$ to $\theta$, we obtain 
\begin{align*}
\|\Lambda^{-1}\theta^*(T)\|^2_{L^2} +\int^T_{0}\|g^*\|^2_{L^2(\Gamma)}\, dt
\leq\|\Lambda^{-1}\theta(T)\|^2_{L^2} +\int^T_{0}\|g\|^2_{L^2(\Gamma)}\, dt,
\end{align*}
for any $g\in L^{2}(0, T; V^0_n(\Gamma))$. Thus, $(g^*,v^*,  \theta^*)$ is an optimal solution to the problem~$(P)$. In particular, if  we set $g=g^*$, then infimum  of $J$ can be reached. 

Step 3: 
 Combining \eqref{convg_rhoH1} with $\sup_{t\in[0, T]}\|\theta^*\|_{L^\infty}=\|\theta_0\|_{L^\infty}$ and \eqref{cong_thetastar}, we can easily show that  
 \begin{align}
\theta^*_{\epsilon}\nabla \rho^*_{\epsilon}\to  \theta^*\nabla \rho^* \ \ \text{weakly in }\ \ L^2(0, T; L^1(\Omega))  \ \ \text{as}\ \  \epsilon\to 0. 
\label{1convg_diff_opt}
\end{align}
However,  $\theta^*_{\epsilon}\nabla \rho^*_{\epsilon}$ also converges to  $\xi$   weakly   in  $L^2(0, T; L^1(\Omega)) $. Thus 
$\xi=\theta^*\nabla \rho^*$, and therefore, 
\begin{align}
g^{*}=\lim_{\epsilon\to 0} g^*_{\epsilon}=-\frac{1}{\gamma}\lim_{\epsilon\to 0}L^{*}(\mathbb{P}(\theta^*_{\epsilon} \nabla \rho^*_{\epsilon}))
=-\frac{1}{\gamma}L^{*}(\mathbb{P}(\theta^* \nabla \rho^*)),\label{2Ori_pt_cond}
\end{align}
which completes the proof.
\end{proof}

 The following theorem shows that any  optimal controller to the problem~$(P)$ can be derived    from the optimality condition of the approximating control problem
  $(P_{\epsilon})$.

\begin{thm} \label{key}
Assume $\theta_0\in L^{\infty}(\Omega)\cap H^1(\Omega)$ and  $v_0\in V^{d/2-1+2\eta}_n(\Omega), d=2,3$. If $(g^*, v^*, \theta^*)$  is an optimal solution to  the problem~$(P)$, then $g^*$ can be solved from \eqref{EQ01}--\eqref{ini}, \eqref{adj}--\eqref{adj_final}, and the optimality condition \eqref{Ori_pt_cond}.
\end{thm}

\begin{proof}
Let 
$(g^*, v^*, \theta^*)$ be any optimal solution to the problem~$(P)$.  We first  employ the idea  as in  \cite[Theorem~5]{barbu2016optimal} to impose  a penalization on the cost functional $J_{\epsilon}$ as to establish the relation  between $(g^*, v^*, \theta^*)$ and the optimal solution to the new defined cost functional. 
Consider  the  minimization problem
\begin{align*}
\min\{J_{\epsilon}(g)+ \frac{1}{2}\int^T_{0}\|g-g^*\|^2_{L^2(\Gamma)}\, dt\}.\quad  (\hat{P}_{\epsilon})
\end{align*}
If we let $(\hat{g}_{\epsilon}, \hat{v}_{\epsilon},\hat{\theta}_{\epsilon})$ be the optimal solution  to the problem~$(\hat{P}_{\epsilon})$, then 
\begin{align}
J_{\epsilon}(\hat{g}_{\epsilon})+ \frac{1}{2}\int^T_{0}\|\hat{g}_{\epsilon}-g^*\|^2_{L^2}\, dt\leq J_{\epsilon}(g)+ \frac{1}{2}\int^T_{0}\|g-g^*\|^2_{L^2}\, dt,
\label{2new_J}
\end{align}
for any $g\in L^{2}(0, T; V^0_n(\Gamma))$. As proven in Theorem~\ref{cong_opt_pair}, there exists a  subsequence, still denoted by $\{(\hat{g}_{\epsilon},\hat{v}_{\epsilon}, \hat{\theta}_{\epsilon})\}$, satisfying   
\begin{align}
&\hat{g}_{\epsilon}\to \hat{g}^*\ \ \text{strongly in} \  \ L^{2}(0, T; V^{0}_{n}(\Gamma)), \ \ \text{as}\ \ \epsilon\to0, \label{2cong_g}\\
&\hat{v}_{\epsilon}\to \hat{v}^* \  \ \text{strongly in} \ \   L^{2}(0, T; V^{0}_{n}(\Omega)),  \ \ \text{as}\ \ \epsilon\to0,\nonumber\\
&\hat{\theta}_{\epsilon}\to \hat{\theta}^*  \ \ \text{strongly in}  \ \ L^{2}(\Omega), \ \text{uniformly in}\  [0, T], \ \ \text{as}\ \ \epsilon\to0.
 \nonumber
\end{align}
By  the  weakly lower semicontinuity of norms, we can pass  to the limit in \eqref{2new_J}  and obtain 
\begin{align}
J(\hat{g}^*)+ \frac{1}{2}\int^T_{0}\|\hat{g}^*-g^*\|^2_{L^2(\Gamma)}\, dt\leq J(g)+ \frac{1}{2}\int^T_{0}\|g-g^*\|^2_{L^2(\Gamma)}\, dt,
\end{align}
for all $g\in L^{2}(0, T; V^0_{n}(\Gamma))$. In particular, setting $g=g^*$ yields 
\begin{align}
J(\hat{g}^*)+ \frac{1}{2}\int^T_{0}\|\hat{g}^*-g^*\|^2_{L^2(\Gamma)}\, dt\leq J(g^*),
\end{align}
which  indicates  
\[\int^T_{0}\|\hat{g}^*-g^*\|^2_{L^2}\, dt=0.\]
Therefore,  $g^*=\hat{g}^*$, and hence $v^*=\hat{v}^*$ and $\theta^*=\hat{\theta}^*$. Moreover, according to \eqref{2cong_g} we  get
\begin{align}
&\hat{g}_{\epsilon}\to g^*\quad \text{strongly in} \   L^{2}(0, T; V^{0}_{n}(\Gamma)),  \ \ \text{as}\ \ \epsilon\to0. \label{3cong_g}
\end{align}
Following the proof of  Theorem~\ref{thm_App_opt_cond},  we have   the optimality condition for the problem~$(\hat{P}_{\epsilon})$  given by
\begin{align}
\gamma \hat{g}^*_{\epsilon}+ \hat{g}^*_{\epsilon}-g^*=-L^{*}\mathbb{P}(\hat{\theta}_{\epsilon}\nabla {\hat{\rho}_{\epsilon}}), \label{App_hat_optcond}
\end{align}
where $ \hat{\rho}_{\epsilon}$ satisfies  \eqref{App_rho}--\eqref{App_rhoT}.  Letting $\epsilon\to 0$, we obtain  from  \eqref{2Ori_pt_cond}  and 
 \eqref{3cong_g} that
\begin{align*}
g^*=-\frac{1}{\gamma }L^{*}\mathbb{P}(\theta^*\nabla \rho^*), 
\end{align*}
which completes the proof.
\end{proof}

\section{Uniqueness of the optimal controller to $(P)$ for $d=2$}
\label{uniqueness}
In this section, we present the  uniqueness of  the optimal controller to the problem~$(P)$ for $d=2$ and $\gamma$ sufficiently large.
In this case  we set $ 0<\eta<1/4$. The main 
result is given by the following theorem. 

\begin{thm} \label{orig_uniq} 
Assume  $\theta_0\in L^{\infty}(\Omega)\cap H^1(\Omega)$, $v_0\in V^{2\eta}_n(\Omega)$ for  $d=2$, and $\gamma$  sufficiently large, there exists at most  one optimal controller $g\in U_{ad}$ to the problem~$(P)$, which is given by \eqref{Ori_pt_cond}.
\end{thm}

\begin{proof}
Assume  that there are two pair of optimal solutions  to the problem~$(P)$, denoted by  $(g_{i}, \theta_{i}, v_i), i=1,2$. Then from \eqref{Ori_pt_cond}, Lemma \ref{App_EST_v}, and \eqref{1EST_te_H1} we have
\[g_i\in L^{2}(0, T; V^{3/2}_{n}(\Gamma))\cap H^{3/4}(0, T; V^{0}_{n}(\Gamma)),\]
\begin{align}
\int^T_{0}\|\nabla v_i\|_{L^\infty}\leq C( v_0, g_i, T), \quad \text{and}\quad \sup_{t\in[0,T]}\|\nabla \theta_i\|_{L^2}\leq C(\theta_{0}, v_0, g_i, T).
\label{0EST_uniq}
\end{align}
The corresponding solutions to the adjoint problem \eqref{App_rho}--\eqref{App_rhoT} are denoted by $\rho_{i}, i=1,2$.
 Then $G=g_{1}-g_{2}$, $\vartheta=\theta_{1}-\theta_{2}$, $LG=v_1-v_2$, and $\varrho=\rho_{1}-\rho_{2}$
satisfy the  system 
  \begin{align}
   & \frac{\partial \vartheta}{\partial t}+ (LG)\cdot \nabla\theta_{1}+v_2\cdot \nabla\vartheta=0, \quad
    \vartheta(0)=0,  \label{diff_theta1}\\
    &-\frac{ \partial{\varrho}}{\partial t}-(LG)\cdot \nabla  \rho_{1}-v_2\nabla\varrho=0,  \quad  
    \varrho(T) =\Lambda^{-2}\vartheta(T),\label{diff_varho1}
      \end{align}
      and 
        \begin{align}
 G=-\frac{1}{\gamma}L^{*}(\mathbb{P}(\vartheta \nabla \rho_{1}+\theta_{2}\nabla \varrho)). \label{EST_G}
      \end{align}
Applying  $L^2$-estimate on $\vartheta$ gives
\begin{align*}
  & \frac{d \|\vartheta\|^2_{L^2}}{2d t}= \int_{\Omega} (LG)\cdot \nabla \theta_{1} \vartheta\, dx
  \leq\|LG\|_{L^\infty}\|\nabla \theta_{1} \|_{L^2} \|\vartheta\|_{L^2},
  \end{align*}
from where
  \begin{align}
  \sup_{t\in [0, T]}  \|\vartheta\|_{L^2}
 &\leq \int^T_{0}  \|LG\|_{L^\infty}\|\nabla \theta_{1} \|_{L^2}\, dt
 \leq  \sup_{t\in [0, T]} \|\nabla \theta_{1} \|_{L^2} \int^T_{0}  \|LG\|_{L^\infty}\, dt.  \label{EST_vartheta}
 \end{align}
By \eqref{EST_G}, \eqref{5L_oper}  and \eqref{2L_adj}, we get
 \begin{align}
&\int^T_{0}  \|LG\|_{L^\infty}\, dt\leq C
 \int^T_{0}  \|L\frac{1}{\gamma}L^*\mathbb{P}(\vartheta \nabla \rho_{1}+\theta_{2}\nabla \varrho)\|_{H^{1+\varepsilon}(\Omega)}\, dt, 
 \quad 0<\varepsilon\leq 1/2, \nonumber \\
&\qquad\leq C\frac{1}{\gamma} \sqrt{T}\|LL^*\mathbb{P}(\vartheta \nabla \rho_{1}+\theta_{2}\nabla \varrho)\|_{L^2(0, T; H^{1+\varepsilon}(\Omega))}
 \nonumber \\
&\qquad\leq  C\frac{1}{\gamma} \sqrt{T}\|L^*\mathbb{P}(\vartheta \nabla \rho_{1}+\theta_{2}\nabla \varrho)\|_{L^2(0, T; L^2(\Gamma))}  \label{1EST_LG_P} \\
&\qquad\leq C\frac{1}{\gamma} \sqrt{T}\|\mathbb{P}^*(\vartheta \nabla \rho_{1}+\theta_{2}\nabla \varrho)\|_{L^2(0, T; (H^{ 1+\varepsilon}(\Omega))')} \label{2EST_LG_P} \\
& \qquad
  \leq  C\frac{1}{\gamma} \sqrt{T}\left(  \|\mathbb{P}^*(\vartheta \nabla \rho_{1})\|_{L^2(0, T; (H^{1+\varepsilon}(\Omega))')}
  +\| \mathbb{P}^*(\theta_{2}\nabla \varrho)\|_{L^2(0, T; (H^{1+\varepsilon}(\Omega))')}\right), \label{0EST_LG}
  \end{align}
 From \eqref{1EST_LG_P} to \eqref{2EST_LG_P} we  used the property of $L^*$ given by \eqref{3L_adj} and replaced $\mathbb{P}$ by $\mathbb{P}^*$ due to Remark \ref{projector}. For the first term on the right hand side of \eqref{0EST_LG} we use  duality \eqref{duality_P} and obtain 
 \begin{align}
& \int^T_{0} \| \mathbb{P}^*(\vartheta \nabla \rho_{1})\|^2_{(H^{1+\varepsilon}(\Omega))'}\, dt
  =\int^T_{0}\left(\sup_{\psi\in H^{1+\varepsilon}(\Omega)}\frac{|\int_{\Omega }(\vartheta \nabla \rho_{1})\cdot (\mathbb{P} \psi)\, dx|}{\|\psi\|_{H^{1+\varepsilon}}}\right)^2\, dt, \nonumber\\
  &\qquad\leq C \int^T_{0}\left( \sup_{\psi\in H^{1+\varepsilon}(\Omega)}\frac{\|\vartheta\|_{L^2} \|\nabla \rho_{1}\|_{L^2} \|\psi\|_{L^\infty}}{\|\psi\|_{H^{1+\varepsilon}}}\right)^2\, dt
  \label{0fail_d3}\\
  &\qquad\leq C\int^T_{0}\left(\sup_{\psi\in H^{1+\varepsilon}(\Omega)}\frac{\|\vartheta\|_{L^2} \|\nabla \rho_{1}\|_{L^2} \|\psi\|_{H^{1+\varepsilon}}}
  {\|\psi\|_{H^{1+\varepsilon}}} \right)^2\, dt \label{1fail_d3} \\
  &\qquad \leq CT(\sup_{t\in [0, T]}\|\vartheta\|_{L^2} )^2 (\sup_{t\in [0, T]} \|\nabla \rho_{1}\|_{L^2})^2.
    \label{2EST_vartheta}
  \end{align}

To estimate the second term on the right hand of \eqref{1EST_LG}, we have
   \begin{align}
& \int^T_{0} \|\mathbb{P}^*(\theta_{2}\nabla \varrho)\|^2_{(H^{1+\varepsilon}(\Omega))'}\, dt
  =\int^T_{0}\left(\sup_{\psi\in H^{1+\varepsilon}(\Omega)}\frac{|\int_{\Omega }(\theta_{2}\nabla \varrho)\cdot (\mathbb{P}\psi)\, dx|}{\|\psi\|_{H^{1+\varepsilon}}}\right)^2\, dt,\nonumber\\
  &\qquad =\int^T_{0}\left(\sup_{\psi\in H^{1+\varepsilon}(\Omega)}\frac{|\int_{\Omega }\theta_{2}\nabla\cdot ( \varrho (\mathbb{P}\psi))\, dx
  -\int_{\Omega }\theta_{2}\varrho\nabla \cdot (\mathbb{P}\psi)\, dx|}{\|\psi\|_{H^{1+\varepsilon}}}\right)^2 \, dt\nonumber\\
  &\qquad =\int^T_{0}\left(\sup_{\psi\in H^{1+\varepsilon}(\Omega)}\frac{|\int_{\Gamma }\theta_{2} \varrho(\mathbb{P}\psi)\cdot n\, dx
  -\int_{\Omega }\nabla \theta_{2} \cdot (\varrho(\mathbb{P}\psi))\, dx|}{\|\psi\|_{H^{1+\varepsilon}}}\right)^2\, dt\nonumber\\
 &\qquad \leq C \int^T_{0} \left(\sup_{\psi\in H^{1+\varepsilon}(\Omega)}\frac{\|\nabla \theta_{2}\|_{L^2} \|\varrho\|_{L^2}\|\psi\|_{L^{\infty}}
 }{\|\psi\|_{H^{1+\varepsilon}}}\right)^2\, dt \label{2fail_d3}\\
 &\qquad \leq C \int^T_{0}\left(\sup_{\psi\in H^{1+\varepsilon}(\Omega)}\frac{\|\nabla \theta_{2}\|_{L^2} \|\varrho\|_{L^2}\|\psi\|_{H^{1+\varepsilon}}}
 {\|\psi\|_{H^{1+\varepsilon}}}\right)^2\, dt\label{3fail_d3}\\
  &\qquad \leq CT(\sup_{t\in [0, T]}\|\nabla \theta_2\|_{L^2})^2 ( \sup_{t\in [0, T]}\|\varrho\|_{L^2})^2.
   \label{3EST_vartheta}
  \end{align}
Combining  \eqref{0EST_LG}  with \eqref{2EST_vartheta}--\eqref{3EST_vartheta}  yields
\begin{align}
\int^T_{0}  \|LG\|_{L^\infty}\, dt\leq &
C\frac{1}{\gamma}T\left(\sup_{t\in [0, T]}\|\vartheta\|_{L^2}  \sup_{t\in [0, T]} \|\nabla \rho_{1}\|_{L^2}
+\sup_{t\in [0, T]}\|\nabla \theta_2\|_{L^2} \sup_{t\in [0, T]}\|\varrho\|_{L^2}\right). \label{1EST_LG}
\end{align}
It remains  to estimate 
$\|\varrho\|_{L^2(\Omega)}$.
Let  $\tilde{\varrho}(t)=\varrho(T-t)$, $t\in [0, T]$. 
Then $\tilde{\varrho}(t)$ satisfies 
  \begin{align}
   &\frac{ \partial{\tilde{\varrho}}}{\partial t}-(LG(T-t))\cdot \nabla  \rho_1(T-t)-v_2(T-t)\nabla\tilde{\varrho}=0, \label{back_rho1}\\
  &  \tilde{\varrho}(0)= \Lambda^{-2}\vartheta(T). \label{2back_rho1}
      \end{align}
      Applying $L^2$-estimate for  $\tilde{\varrho}$ yields 
      \begin{align}
\sup_{t\in [0, T]}  \|\tilde{\rho}\|_{L^2}
&\leq \int^T_{0}  \|LG(T-t)\|_{L^\infty}\|\nabla \rho_1 (T-t)\|_{L^2}\, dt+\|\tilde{\rho}(0)\|_{L^2}\nonumber\\
&\leq \sup_{t\in [0, T]}\|\nabla \rho_1 \|_{L^2}\int^T_{0}  \|LG\|_{L^\infty}\, dt+C\|\vartheta(T)\|_{L^2}.
\label{EST_varrho}
 \end{align}
Now combining  \eqref{1EST_LG} with   \eqref{0EST_uniq},   \eqref{EST_vartheta}, and  \eqref{EST_varrho}  gives
\begin{align}
&\int^T_{0}  \|LG\|_{L^\infty}
\leq C(\theta_{0}, v_0, g_1, g_2, T)\frac{1}{\gamma}
( \sup_{t\in [0, T]}\|\vartheta\|_{L^2} +\sup_{t\in [0, T]}\|\varrho\|_{L^2} )\nonumber\\
&\quad\leq C(\theta_{0}, v_0, g_1, g_2, T)\frac{1}{\gamma}
\bigg[\sup_{t\in [0, T]} \|\nabla \theta_{1} \|_{L^2}\int^T_{0}  \|LG\|_{L^\infty}\, dt\nonumber\\ 
&\qquad+\big(\sup_{t\in [0, T]}\|\nabla \rho_1 \|_{L^2}\int^T_{0}  \|LG\|_{L^\infty}\, d\tau
+\sup_{t\in [0, T]} \|\nabla \theta_{1} \|_{L^2}\int^T_{0}  \|LG\|_{L^\infty}\, dt \big)\bigg]\nonumber\\
&\quad\leq C(\theta_{0}, v_0, g_1, g_2, T)\frac{1}{\gamma}
\int^T_{0}  \|LG\|_{L^\infty}\, d\tau.  \label{EST_LG}
\end{align}
If we let $\gamma$ be sufficiently large so that
\[C(\theta_{0}, v_0, g_1, g_2, T)\frac{1}{\gamma}<1\quad \text{or}\quad \gamma> C(\theta_{0}, v_0, g_1, g_2, T),\]
then
\begin{align}
\int^T_{0}  \|LG\|_{L^\infty}\, dt=0.\label{EST_LG1}
\end{align}
Lastly, by the linearity of $L$, we derive  that $G=0$. 
 Uniqueness of the optimal solution   for large $\gamma$ and $d=2$ is established.

\end{proof}

\begin{remark}
The uniqueness for $d=3$ can not be carried out by the current approach due to the failure from  \eqref{0fail_d3} to \eqref{1fail_d3} and from 
\eqref{2fail_d3}  to \eqref{3fail_d3}. This is because when $d=3$, the regularity of the  test function $\psi$  can not go beyond  $H^{3/2-\epsilon}$, where  $\varepsilon$ is arbitrarily small. Therefore,  the $L^\infty$-norm of $\psi$ in the numerators of   \eqref{0fail_d3} and  \eqref{2fail_d3} can not be bounded.
 \end{remark}


%

\section{Conclusions}

Compared to the optimality system  presented in  \cite[Theorem~4.1]{hu2017boundary}, the current approach of constructing   an approximating control problem  provides a  much more transparent result. In addition,   uniqueness of the optimal controller can be derived for $d=2$.
These  will greatly contribute to implementing the solution   by employing  the gradient based iterative   schemes  in our future work.\\

\noindent \textbf{Acknowledgments}\,\, The author would like to thank  Irena Lasiecka for
her valuable suggestions which improved  the paper.
The author was partially supported by the  NSF grant DMS-1813570,  the DIG and FY 2018 ASR+1 Program at the Oklahoma State University.

\end{document}